%% file: EcyclicBCD_submit2_v3.tex
\documentclass[final]{siamltex}
\usepackage[left=1in,top=1in,right=1in,bottom=1in,letterpaper]{geometry}

\usepackage{amsmath,amssymb,amsfonts}
\usepackage{dsfont}
\usepackage{cite}
\usepackage{multirow}

\usepackage[ruled,vlined,linesnumbered]{algorithm2e}
\usepackage{subfigure}
\usepackage[usenames]{color}
\usepackage[final,notref,notcite]{showkeys}

\newif\ifpdf\ifx\pdfoutput\undefined\pdffalse\else\pdfoutput=1\pdftrue\fi
\usepackage[colorlinks=true, pdfstartview=FitV, linkcolor=blue,
                                    citecolor=blue, urlcolor=blue]{hyperref}

  \ifpdf        \usepackage[pdftex]{graphicx}
        \usepackage{epstopdf}
        \usepackage{url}
  \else\usepackage{graphicx}\fi

\newtheorem{condition}{Condition}[section]
\newtheorem{remark}{Remark}[section]
\newtheorem{assumption}{Assumption}

\input{macros.tex}

\newcount\refnum\refnum=0
\def\myref{{\global\advance\refnum by 1} {\bf \large Lecture \the \refnum. }}

\begin{document}

\title{A globally convergent algorithm for nonconvex optimization based on block coordinate update\thanks{The work is supported in part by NSF DMS-1317602 and ARO MURI W911NF-09-1-0383.}
}
\author{Yangyang Xu\thanks{\url{yangyang.xu@uwaterloo.ca}. Department of Combinatorics and Optimization, University of Waterloo, Waterloo, Canada.
}
\and
Wotao Yin\thanks{\url{wotaoyin@math.ucla.edu}. Department of Mathematics, UCLA, Los Angeles, California, USA.}
}

\date{\today}

\maketitle

\begin{abstract} Nonconvex optimization arises in many areas of computational science and engineering. 
However, most nonconvex optimization  algorithms are only known to have local convergence or subsequence convergence properties. 
In this paper, we propose an algorithm for  nonconvex optimization and establish its global convergence (of the whole sequence) to a critical point. In addition, we give its asymptotic convergence rate and numerically demonstrate its efficiency.

In our algorithm, the variables of the underlying problem are either treated  as one block or  multiple disjoint blocks.  It is assumed that each non-differentiable component of the objective function, or each constraint, applies only to one block  of variables. The differentiable
components of the objective function, however, can involve  multiple blocks  of variables together.

Our algorithm updates one block of variables at a time by minimizing a certain prox-linear surrogate, 
{along with an extrapolation to accelerate its convergence.} The order of update can be either deterministically cyclic or randomly shuffled for each cycle. In fact, our convergence analysis only needs that each block  be updated at least once in every fixed number of iterations. We show its global convergence (of the whole sequence) to a critical point under fairly loose conditions including, in particular, the  Kurdyka-{\L}ojasiewicz (KL) condition, which is satisfied by a broad class of nonconvex/nonsmooth applications. These results, of course, remain valid when the underlying problem is convex.

We apply our convergence results to the coordinate descent iteration for non-convex regularized linear regression, as well as a modified rank-one residue iteration  for nonnegative matrix factorization. We show that both  applications have global convergence. Numerically, we tested our algorithm on nonnegative matrix and tensor factorization problems, where  random shuffling clearly improves to chance to avoid low-quality  local solutions.
\end{abstract}
\begin{keywords}
nonconvex optimization, nonsmooth optimization, block coordinate descent, Kurdyka-{\L}ojasiewicz inequality, prox-linear,  whole sequence convergence
\end{keywords}

\section{Introduction}

In this paper, we consider (nonconvex) optimization problems in the form of
\begin{equation}\label{eq:main}
\begin{split}
&\underset{\bfx}{\mbox{minimize}}~F(\bfx_1,\cdots,\bfx_s)\equiv f(\bfx_1,\cdots,\bfx_s)+\sum_{i=1}^s r_i(\bfx_i),\\[-0.1cm]
&\mbox{subject to}~ \bfx_i\in \cX_i,~i=1,\ldots ,s,
\end{split}
\end{equation}
where variable $\bfx=(\bfx_1,\cdots,\bfx_s)\in\RR^n$ has $s$ blocks, $s\ge 1$, function $f$ is continuously differentiable,  functions $r_i$,  $i=1,\cdots,s$, are proximable\footnote{A function $f$ is proximable if it is easy to obtain the minimizer of $f(x)+\frac{1}{2\gamma}\|x-y\|^2$ for any input $y$ and $\gamma>0$.} but not necessarily differentiable. It is standard to assume that both $f$ and $r_i$ are closed and proper and  the sets $\cX_i$ are closed and nonempty. Convexity is \emph{not} assumed for $f$, $r_i$, or $\cX_i$. By allowing $r_i$ to take the $\infty$-value, $r_i(\bfx_i)$ can incorporate the constraint $\bfx_i\in\cX_i$ since enforcing the constraint is equivalent to minimizing the indicator function of $\cX_i$, and $r_i$ can remain  proper and closed. Therefore, in the remainder of this paper, we do not include the constraints $\bfx_i\in\cX_i$. The functions $r_i$ can incorporate  regularization functions, often used to enforce certain properties or structures in $\vx_i$, for example, the nonconvex $\ell_p$ quasi-norm, $0\le p< 1$, which promotes solution sparsity.

Special cases of \eqref{eq:main} include the following nonconvex problems: $\ell_p$-quasi-norm ($0\le p< 1$) regularized sparse regression problems \cite{natarajan1995sparse, blumensath2009iterative, lai2013improved},  sparse dictionary learning \cite{aharon2006rm, mairal2009online, xu2014patch}, 
matrix rank minimization  \cite{recht2010guaranteed}, matrix factorization with nonnegativity/sparsity/orthogonality regularization \cite{paatero1994positive, lee1999learning, hoyer2004non}, (nonnegative) tensor decomposition \cite{welling2001positive, kolda2009tensor}, and (sparse) higher-order principal component analysis \cite{allen2012sparse}. 

Due to the lack of convexity, standard analysis tools such as convex inequalities and Fej\'er-monotonicity cannot be applied to establish the convergence of the iterate sequence. The case becomes more difficult when the problem is nonsmooth. In these cases, convergence analysis of existing algorithms is typically limited to  objective convergence (to a possibly non-minimal value) or the convergence of a certain  subsequence of iterates to a critical point. (Some exceptions will be reviewed below.) Although whole-sequence convergence is almost always observed, it is rarely proved. This deficiency abates some  widely used algorithms. For example, KSVD  \cite{aharon2006rm}  only has nonincreasing monotonicity of its objective sequence, and iterative reweighted algorithms for sparse and low-rank recovery in \cite{chartrand2008iteratively, mohan2012iterative, lai2013improved} only has subsequence convergence.
Some other methods establish whole sequence convergence by assuming stronger conditions such as local  convexity (on at least a part of the objective) and either unique or isolated limit points, which may be difficult to satisfy or to verify. In this paper, we aim to establish whole sequence convergence with conditions that are provably satisfied by a wide class of functions. 

{Block coordinate descent (BCD) (more precisely, block coordinate update) is very general and widely used for solving both convex and nonconvex problems in the form of \eqref{eq:main} with multiple blocks of variables. 
Since only one block is updated at a time, it has a low per-iteration cost and small memory footprint. Recent literature \cite{nesterov2012efficiency, richtarik2012iteration, saha2013nonasymptotic, lu2013complexity, beck2013convergence, hong2013iteration} 
has found BCD as a viable approach for ``big data'' problems. 
}


\subsection{Proposed algorithm} 
In order to solve \eqref{eq:main}, we propose a block prox-linear (BPL) method, which updates a block of variables at each iteration by minimizing a prox-linear surrogate function.
Specifically, at iteration $k$, a block $b_k\in\{1,\ldots,s\}$ is selected and $\bfx^k=(\bfx_1^k,\cdots,\bfx_s^k)$ is updated as follows:
\begin{equation}\label{eq:ebpg}\left\{
\begin{array}{ll}\bfx_i^{k}=\bfx_i^{k-1},&\text{ if } i\neq b_k,\\[0.2cm]
\bfx_i^{k}\in\underset{\bfx_i}{\argmin}\,\langle \nabla_{\vx_i} f(\bfx_{\neq i}^{k-1},\hat{\bfx}_i^k), \bfx_i-\hat{\bfx}^k_i\rangle+\frac{1}{2\alpha_k}\|\bfx_i-\hat{\bfx}^k_i\|^2+r_{i}(\bfx_i),&\text{ if } i= b_k,
\end{array}\right. \quad \mbox{for}~i=1,\ldots,s,
\end{equation}
where $\alpha_k>0$ is a stepsize and  $\hat{\vx}_i^k$ is the extrapolation\begin{equation}\label{eq:extrp0}\hat{\vx}_i^k={\vx}_i^{k-1}+\omega_k({\vx}_i^{k-1}
-{\vx}_i^{\prev}),
\end{equation}
where $\omega_k\ge0$ is an extrapolation weight and ${\vx}_i^{\prev}$ is the value of $\bfx_i$ before it was updated to  ${\vx}_i^{k-1}$. 
The framework of our method is given in Algorithm \ref{alg:ebpg}. At each iteration $k$, only the block $b_k$ is updated.

\begin{algorithm}\caption{Randomized/deterministic block prox-linear (BPL) method for problem \eqref{eq:main}}\label{alg:ebpg}
\DontPrintSemicolon
{\small
{\bf Initialization:} $\bfx^{-1}=\bfx^0$.\;
\For{$k=1,2,\cdots$}{
Pick $b_k\in\{1,2,\ldots,s\}$ in a deterministic or random manner.
\;
Set $\alpha_k,\,\omega_k$ and let $\bfx^{k}\gets$ \eqref{eq:ebpg}.\;
\If{stopping criterion is satisfied}{
Return $\bfx^{k}$.\;
}
}
}
\end{algorithm}

{While we can simply set $\omega_k=0$, appropriate $\omega_k>0$ can speed up the convergence; we will demonstrate this in the numerical results below. We can set the stepsize $\alpha_k=\frac{1}{\gamma L_k}$ with any $\gamma>1$, where $L_k>0$ is the Lipschitz constant of $\nabla _{\vx_i}f(\vx_{\neq i}^{k-1}, \vx_i)$ about $\vx_i$. When $L_k$ is unknown or difficult to bound, we can apply backtracking on $\alpha_k$ under the criterion: \begin{equation*}f(\vx^{k})\le f(\vx^{k-1})+\langle\nabla_{\vx_i}f(\vx^{k-1}),\vx_i^{k}-\vx_i^{k-1}\rangle +\frac{1}{2\gamma\alpha_k}\|\vx_i^{k}-\vx_i^{k-1}\|^2.
\end{equation*}}

\subsection*{Special cases}
When there is only one block, i.e., $s=1$,  Algorithm \ref{alg:ebpg} reduces to the well-known (accelerated) proximal gradient method (e.g., \cite{NesterovConvexBook2004, BeckTeboulle2009, G-L-2015apg-noncvx}). When the update block cycles from  1 through $s$, 
 Algorithm \ref{alg:ebpg} reduces to the cyclic block proximal gradient (Cyc-BPG) method in \cite{xu2013block, beck2013convergence}. We can also randomly shuffle the $s$ blocks at the beginning of each cycle. 
We demonstrate in section \ref{sec:example} that random shuffling leads to better numerical performance. When the update block is randomly selected following  the probability $p_i>0$, where $\sum_{i=1}^sp_i=1$, Algorithm \ref{alg:ebpg} reduces to the randomized block coordinate descent method (RBCD) (e.g., \cite{nesterov2012efficiency, richtarik2012iteration, lu2013complexity, lu2013randomized}).  
{
Unlike these existing results, we do not assume convexity.
}

In our analysis,  we impose an essentially cyclic assumption --- each block is selected for update at least once within every $T\ge s$ consecutive iterations --- otherwise  the order is arbitrary.
{Our convergence results apply to all the above special cases except RBCD, whose convergence analysis requires different strategies; see \cite{nesterov2012efficiency, richtarik2012iteration, lu2013complexity} for the convex case and \cite{lu2013randomized} for the nonconvex case.}

\subsection{Kurdyka-{\L}ojasiewicz property}\label{sec:KL}
To establish whole sequence convergence of Algorithm \ref{alg:ebpg}, a key assumption is the Kurdyka-{\L}ojasiewicz (KL) property of the objective function $F$.

A lot of functions are known to satisfy the KL property. Recent works \cite[section 4]{attouch2010proximal} and \cite[section 2.2]{xu2013block} give many specific examples that satisfy the property, such as the $\ell_p$-(quasi)norm $\|\vx\|_p$ with $p\in [0,+\infty]$, any piecewise polynomial functions, indicator functions of polyhedral set, orthogonal matrix set, and positive semidefinite cone, matrix rank function, and so on.
\begin{definition}[Kurdyka-{\L}ojasiewicz property]\label{def:KL}
A function $\psi(\bfx)$ satisfies the KL property at point $\bar{\bfx}\in\mathrm{dom}(\partial\psi)$ if there exist $\eta>0$, a neighborhood $\cB_\rho(\bar{\bfx})\triangleq\{\vx:\|\vx-\bar{\vx}\|<\rho\}$, and a concave function 
$\phi(a)=c\cdot a^{1-\theta}$ for some $c>0$ and $\theta\in[0,1)$ such that the KL inequality holds
\begin{equation}\label{eq:KL}
\phi'(|\psi(\bfx)-\psi(\bar{\bfx})|)\mathrm{dist}(\mathbf{0},\partial \psi(\bfx))\ge 1,\text{ for any }\bfx\in \cB_\rho(\bar{\vx})\cap \mathrm{dom}(\partial \psi) \text{ and }\psi(\bar{\bfx})<\psi(\bfx)< \psi(\bar{\bfx})+\eta,
\end{equation}
where $\mathrm{dom}(\partial\psi)=\{\bfx: \partial\psi(\bfx)\neq\emptyset\}$ 
and $\mathrm{dist}(\mathbf{0},\partial \psi(\bfx))=\min\{\|\bfy\|:\bfy\in\partial \psi(\bfx)\}$.
\end{definition}

The KL property was introduced by {\L}ojasiewicz \cite{lojasiewicz1993geometrie} for real analytic functions. 
Kurdyka \cite{kurdyka1998gradients} extended it to 
functions of the $o$-minimal structure. Recently, the KL inequality \eqref{eq:KL} was  further 
extended to nonsmooth sub-analytic functions \cite{bolte2007lojasiewicz}. The work \cite{bolte2010characterizations} characterizes the geometric meaning of the KL inequality.

\subsection{Related literature}
There are many methods that solve general nonconvex problems. Methods in the papers \cite{fuduli2004minimizing, burke2005robust, chen2012smoothing, bagirov2013subgradient}, the books \cite{Bertsekas-NLP, NocedalWright06}, and in the references therein,  do not break variables into blocks. They usually have the properties of local convergence or subsequence convergence to a critical point, or global convergence in the terms of the violation of optimality conditions. Next, we review BCD methods.

BCD has been extensively used in many applications. Its original form, block coordinate minimization (BCM), which updates a block  by minimizing the original objective with respect to that block, dates back to the 1950's \cite{Hildreth-57} and is closely related to the Gauss-Seidel and SOR methods for  linear equation systems. Its convergence was studied under a variety of settings (cf. \cite{GrippoSciandrone1999, Tseng-01, razaviyayn2013unified} and the references therein). The convergence rate of BCM was  established under  the strong convexity assumption \cite{LuoTseng1992} for the multi-block case and  under the general convexity assumption \cite{beck2013convergence} for the two-block case.
To have even cheaper updates, one can update a block  approximately, for example, by minimizing an approximate objective like was done in \eqref{eq:ebpg}, instead of sticking to the original objective. The work \cite{TsengYun2009} 
is a block coordinate gradient descent (BCGD) method 
where taking a block gradient step is equivalent to minimizing a certain prox-linear approximation of the objective.
Its whole sequence convergence and local convergence rate were established under the assumptions of a so-called \emph{local Lipschitzian error bound} and the convexity of the objective's nondifferentiable part. 
The randomized block coordinate descent (RBCD) method in \cite{nesterov2012efficiency, lu2013randomized}  randomly chooses the block to update 
 at each iteration and is not essentially cyclic. 
Objective convergence was established \cite{nesterov2012efficiency, richtarik2012iteration}, and the violation of the first-order optimization condition was shown to converge to \emph{zero} \cite{lu2013randomized}. There is no iterate  convergence result for RBCD. 

Some special cases of Algorithm \ref{alg:ebpg} have been analyzed in the literature. The work \cite{xu2013block} uses cyclic updates of a fixed order and assumes block-wise convexity; 
\cite{bolte2013proximal} studies two blocks without extrapolation, namely, $s=2$ and $\hat{\vx}_{i}^k=\vx_{i}^{k-1},\,\forall k$ in \eqref{eq:ebpg}. A more general result is \cite[Lemma 2.6]{attouch2013convergence}, where three conditions for whole sequence convergence are given and are met by methods including averaged projection, proximal point, and forward-backward splitting. Algorithm \ref{alg:ebpg}, however, does not satisfy the three conditions in \cite{attouch2013convergence}. 

{The extrapolation technique in \eqref{eq:extrp0} has been applied to accelerate the (block) prox-linear method for solving convex optimization problems (e.g., \cite{NesterovConvexBook2004, BeckTeboulle2009, richtarik2012iteration, lu2013complexity}). Recently, \cite{xu2013block, G-L-2015apg-noncvx} show that the (block) prox-linear iteration with extrapolation can still converge if the nonsmooth part of the problem is convex, while the smooth part can be nonconvex. Because of the convexity assumption, their convergence results do not apply to Algorithm \ref{alg:ebpg} for solving the general nonconvex problem \eqref{eq:main}.}



\subsection{Contributions} We summarize the main contributions of this paper as follows.
\begin{itemize}
\item We propose a block prox-linear (BPL) method  for  nonconvex smooth and nonsmooth optimization. 
Extrapolation is used to accelerate it. {To our best knowledge, this is the first work of  prox-linear acceleration for fully nonconvex problems (where both smooth and nonsmooth terms are nonconvex) with a convergence guarantee. However, we have not proved any improved convergence rate.} 
\item Assuming essentially cyclic updates of the blocks, we obtain the whole sequence convergence of BPL to a critical point with rate estimates, by first establishing subsequence convergence and then applying  the Kurdyka-{\L}ojasiewicz (KL) property. Furthermore, we tailor our convergence analysis to several existing algorithms, including non-convex regularized linear regression and  nonnegative matrix factorization, to improve their existing convergence results.
\item We numerically tested  BPL  on nonnegative matrix and tensor factorization problems. At each cycle of updates, the blocks were randomly shuffled. We observed that BPL was very efficient and that random shuffling avoided local solutions more effectively than the deterministic cyclic order.
\end{itemize}

\subsection{Notation and preliminaries}
We restrict our discussion in $\RR^n$ equipped with the Euclidean norm, denoted by $\|\cdot\|$. However, all our results extend to general of primal and dual norm pairs. The lower-case letter $s$ is reserved for the number of blocks and $\ell, L, L_k,\ldots$ for various Lipschitz constants. $\vx_{<i}$ is short for $(\vx_1,\ldots,\vx_{i-1})$, $\vx_{>i}$ for $(\vx_{i+1},\ldots,\vx_s)$, and $\vx_{\neq i}$ for $(\vx_{<i},\vx_{>i})$. We simplify $f(\vx_{<i},\hat{\vx}_i,\vx_{>i})$ to $f(\vx_{\neq i},\hat{\vx}_i)$. The distance of a point $\vx$ to a set $\cY$ is denoted by $\text{dist}(\vx,\cY)=\inf_{\vy\in\cY}\|\vx-\vy\|.$

Since the update may be aperiodic, extra notation is used for when and how many times a block is updated. Let $\cK[i,k]$ denote the  set of iterations in which the $i$-th block has been selected to update till the $k$th iteration:
\begin{equation}
\label{i-th-iter}\cK[i,k]\triangleq\{\kappa: b_\kappa=i,\, 1\le\kappa\le k\}\subseteq\{1,\ldots,k\},
\end{equation}
and let $$d_i^k\triangleq\big|\cK[i,k]\big|,$$ which is the number of times the $i$-th block has been updated till iteration $k$. For $k=1,\ldots,$ we have  $\cup_{i=1}^s\cK[i,k]=[k]\triangleq\{1,2,\ldots,k\}$ and $\sum_{i=1}^sd_i^k=k$.

Let $\vx^k$ be the value of $\vx$ after the $k$th iteration, and for each block $i$, $\tilde{\bfx}_i^j$ be the value of $\vx_i$ after its $j$th update. By letting $j=d_i^k$, we have $\bfx_i^k = \tilde{\bfx}_i^j$.

The extrapolated point in \eqref{eq:ebpg} (for $i=b_k$) is computed from the last two  updates of the same block:
\begin{equation}\label{eq:extrap}\hat{\bfx}_{i}^{k}=\tilde{\bfx}_{i}^{j-1}+
\omega_k(\tilde{\bfx}_{i}^{j-1}-\tilde{\bfx}_{i}^{j-2}),\text{ where }j=d_i^k,
\end{equation} for some weight $0\le\omega_k\le 1$. We partition the set of Lipschitz constants and  the extrapolation weights into $s$ disjoint subsets as
\begin{subequations}
\begin{align}
\{L_\kappa:1\le \kappa\le k\}=\cup_{i=1}^s\{L_\kappa:\kappa\in\cK[i,k]\}\triangleq\cup_{i=1}^s\{\tilde{L}_i^j: 1\le j\le d_i^k\},\label{setL}\\
\{\omega_\kappa:1\le \kappa\le k\}=\cup_{i=1}^s\{\omega_\kappa:\kappa\in\cK[i,k]\}\triangleq\cup_{i=1}^s\{\tilde{\omega}_i^j: 1\le j\le d_i^k\}.\label{setw}
\end{align}
\end{subequations}
Hence, for each block $i$, we have three sequences: 
\begin{subequations}\label{seq}
\begin{align}
&\text{value of }\vx_i:\ \tilde{\vx}_i^1,\tilde{\vx}_i^2,\ldots,\tilde{\vx}_i^{d_i^k},\ldots;\label{seq-xi}\\
&\text{Lipschitz constant}:\ \tilde{L}_i^1, \tilde{L}_i^2,\ldots, \tilde{L}_i^{d_i^k},\ldots;\label{seq-Li}\\
&\text{extrapolation weight}:\ \tilde{\omega}_i^1,\tilde{\omega}_i^2,\ldots,\tilde{\omega}_i^{d_i^k},\ldots.\label{seq-wi}
\end{align}
\end{subequations}
{For simplicity, we take stepsizes and extrapolation weights as follows
\begin{equation}\label{eq:alpha-omega}\alpha_k=\frac{1}{2L_k},\,\forall k,\qquad \tilde{\omega}_i^j\le\frac{\delta}{6}\sqrt{\tilde{L}_i^{j-1}/\tilde{L}_i^{j}},\,\forall i,j,\text{ for some }\delta<1.
\end{equation}
However, if the problem \eqref{eq:main} has more structures such as block convexity, we can use larger $\alpha_k$ and $\omega_k$; see Remark \ref{rm:multiconv}.} Table \ref{table:notation} summarizes the notation.
In addition, we initialize $\tilde{\vx}_i^{-1}=\tilde{\vx}_i^0=\vx_i^0,~\forall i$. \begin{table}\caption{Summary of notation}\label{table:notation}
\centering
{\small
\begin{tabular}{c|c}
\hline
Notion & Definition\\\hline\hline
$s$ & the total number of blocks\\
$b_k$ & the update block selected at the $k$-th iteration\\
$\cK[i,k]$ & the set of iterations up to $k$ in which  $\vx_i$  is updated; see \eqref{i-th-iter}\\
$d_i^k$ & $\big|\cK[i,k]\big|$:  the number of updates to  $\vx_i$ within the first $k$ iterations\\
$\vx^k$ & the value of $\vx$ after the $k$-th iteration\\
$\tilde{\vx}_i^j$ & the value of $\vx_i$ after its $j$-th update; see \eqref{seq-xi}\\
$L_k$ & the gradient Lipschitz constant of the update block  at the $k$-th iteration; see \eqref{eq:condL}\\
$\tilde{L}_i^j$ & the gradient Lipschitz constant of block $i$ at its $j$-th update; see \eqref{setL} and \eqref{seq-Li}\\
$\omega_k$ &the  extrapolation weight used at the $k$-th iteration\\
$\tilde{\omega}_i^j$ & the extrapolation weight used at the $j$-th update of   $\vx_i$; see \eqref{setw} and \eqref{seq-wi}
\\\hline
\end{tabular}}
\end{table}

We make the following definitions, which can be found in \cite{rockafellar2009variational}.

\begin{definition}[Limiting Fr{\'e}chet subdifferential\cite{kruger2003frechet}]
A vector $\vg$ is a Fr{\'e}chet subgradient of a lower semicontinuous function $F$ at $\vx\in\dom(F)$ if
$$\liminf_{\vy\to\vx,\vy\neq\vx}\frac{F(\vy)-F(\vx)-\langle \vg,\vy-\vx\rangle}{\|\vy-\vx\|}\ge0.$$
The set of Fr{\'e}chet subgradient of $F$ at $\vx$ is called Fr{\'e}chet subdifferential and denoted as $\hat{\partial} F(\vx)$. If $\vx\not\in\dom(F)$, then $\hat{\partial} F(\vx)=\emptyset$.

The limiting Fr{\'e}chet subdifferential is denoted by ${\partial} F(\vx)$ and defined as
$${\partial} F(\vx)=\{\vg: \text{ there is }\vx_m\to\vx \text{ and } \vg_m\in\hat{\partial} F(\vx_m)\text{ such that }\vg_m\to\vg\}.$$
\end{definition}If $F$ is differentiable\footnote{A function $F$ on $\RR^n$ is differentiable at point $\vx$ if there exists a vector $\vg$ such that $\lim_{\vh\to 0}\frac{|F(\vx+\vh)-F(\vx)-\vg^\top\vh|}{\|\vh\|}=0$} at $\vx$, then $\partial F(\vx)=\hat{\partial} F(\vx)=\{\nabla F(\vx)\}$; see \cite[Proposition 1.1]{kruger2003frechet} for example, and if $F$ is convex, then
$\partial F(\vx)=\{\vg: F(\vy)\ge F(\vx)+\langle \vg, \vy-\vx\rangle,\,\forall \vy\in\mbox{dom}(F)\}.$ We use the limiting subdifferential for general nonconvex nonsmooth functions.
For problem \eqref{eq:main},  it holds that (see \cite[Lemma 2.1]{attouch2010proximal} or \cite[Prop. 10.6,  pp. 426]{rockafellar2009variational})
\begin{equation}\label{eq:cart-prod}
\partial F(\vx)=\{\nabla_{\vx_1}f(\vx)+\partial r_1(\vx_1)\}\times\cdots\times\{\nabla_{\vx_s}f(\vx)+\partial r_s(\vx_s)\},
\end{equation}
where $\cX_1\times\cX_2$ denotes the Cartesian product of $\cX_1$ and $\cX_2$ .
\begin{definition}[Critical point]
A point $\vx^*$ is called a critical point of $F$ if $\vzero\in\partial F(\vx^*)$.
\end{definition}

\begin{definition}[Proximal mapping]
For a proper, lower semicontinuous function $r$, its proximal mapping $\prox_r(\cdot)$ is defined as
$$\prox_r(\vx)=\argmin_\vy \frac{1}{2}\|\vy-\vx\|^2+r(\vy).$$
\end{definition}As $r$ is nonconvex, $\prox_r(\cdot)$ is generally set-valued. Using this notation, the update in \eqref{eq:ebpg} can be written as (assume $i=b_k$)
$$\vx_i^k\in\prox_{\alpha_k r_i}\left(\hat{\bfx}_i^k-\alpha_k\nabla_{\vx_i} f(\bfx_{\neq i}^{k-1},\hat{\bfx}_i^k)\right)$$


\subsection{Organization}
The rest of the paper is organized as follows. Section \ref{sec:analysis} establishes convergence results. Examples and applications are given in section \ref{sec:example}, and finally section \ref{sec:conclusion} concludes this paper.

\section{Convergence analysis}\label{sec:analysis}
In this section, we analyze the convergence of Algorithm \ref{alg:ebpg}. Throughout our analysis, we make the following assumptions. 

\begin{assumption}\label{assump1}
$F$ is proper and lower bounded in $\dom(F)\triangleq\{\vx:F(\vx)<+\infty\}$, $f$ is continuously differentiable, and $r_i$ is proper lower semicontinuous for all $i$. 
Problem \eqref{eq:main} has a critical point ${\bfx}^*$, i.e., $\mathbf{0}\in\partial F({\bfx}^*)$.
\end{assumption}
\begin{assumption}\label{assump2}
Let $i=b_k$. $\nabla_{\vx_{i}} f(\bfx_{\neq i}^{k-1},\bfx_{i})$ has Lipschitz continuity constant $L_k$ with respect to $\bfx_{i}$, i.e.,
\begin{equation}\label{eq:condL}
\|\nabla_{\vx_{i}} f(\bfx_{\neq i}^{k-1},\bfu)-\nabla_{\vx_{i}} f(\bfx_{\neq i}^{k-1},\bfv)\|\le
L_{k}\|\bfu-\bfv\|, \ \forall \bfu,\bfv,
\end{equation}
and there exist constants $0<\ell\le L<\infty,$ such that $\ell\le L_k\le L$ for all $k$.
\end{assumption}

\begin{assumption}[Essentially cyclic block update]\label{assump3}
In Algorithm \ref{alg:ebpg}, within any $T$ consecutive iterations, every block is updated at least one time.
\end{assumption}

Our analysis proceeds with several steps. We first estimate the objective decrease after every iteration (see Lemma \ref{lem:dec0}) and then establish a square summable result of the iterate differences (see Proposition \ref{prop:sqsum}). Through the square summable result, we show a subsequence convergence result that every limit point of the iterates is a critical point (see Theorem \ref{thm:subseq}). Assuming the KL property (see Definition \ref{def:KL}) on the objective function and the following monotonicity condition, we establish whole sequence convergence of our algorithm and also give estimate of convergence rate (see Theorems \ref{thm:global-ebpg} and \ref{thm:rate}).
\begin{condition}[Nonincreasing objective]\label{cond-dec} The weight
$\omega_k$ is chosen so that $F(\vx^k)\le F(\vx^{k-1}),\,\forall k$.
\end{condition}

We will show that a range of nontrivial $\omega_k>0$ always exists to satisfy Condition \ref{cond-dec} under a mild assumption, and thus one can  backtrack  $\omega_k$ to ensure $F(\vx^k)\le F(\vx^{k-1}),\,\forall k$. Maintaining the monotonicity of $F(\vx^k)$ can  significantly improve the numerical performance of the algorithm, as shown in  our numerical results below and also in \cite{o2013adaptive, xu2015hsvm}. Note that  subsequence convergence  does not require this condition.

We begin our analysis with the following lemma. The proofs of all the lemmas and propositions are given in Appendix \ref{app:proof-lem}.

{
\begin{lemma}\label{lem:dec0}
Take $\alpha_k$ and $\omega_k$ as in \eqref{eq:alpha-omega}.
After each iteration $k$, it holds
\begin{align}
F(\bfx^{k-1})-F(\bfx^k)\ge & c_1\tilde{L}_i^{j}\|\tilde{\bfx}_i^{j-1}-\tilde{\bfx}_i^{j}\|^2-c_2\tilde{L}_i^{j}(\tilde{\omega}_i^{j})^2\|\tilde{\bfx}_i^{j-2}-\tilde{\bfx}_i^{j-1}\|^2\label{eq:dec2-0}\\
\ge & c_1\tilde{L}_i^{j}\|\tilde{\bfx}_i^{j-1}-\tilde{\bfx}_i^{j}\|^2-\frac{c_2\tilde{L}_i^{j-1}}{36}\delta^2\|\tilde{\bfx}_i^{j-2}-\tilde{\bfx}_i^{j-1}\|^2,\label{eq:dec2-1}
\end{align}
where $c_1=\frac{1}{4}, c_2=9$, $i=b_k$ and $j=d_i^k$.
\end{lemma}

\begin{remark}\label{rm:large-alpha}
We can relax the choices of $\alpha_k$ and $\omega_k$ in \eqref{eq:alpha-omega}. For example, we can take $\alpha_k=\frac{1}{\gamma L_k},\,\forall k$, and $\tilde{\omega}_i^j\le\frac{\delta(\gamma-1)}{2(\gamma+1)}\sqrt{\tilde{L}_i^{j-1}/\tilde{L}_i^{j}},\,\forall i,j$ for any $\gamma>1$ and some $\delta<1$. Then, \eqref{eq:dec2-0} and \eqref{eq:dec2-1} hold with $c_1=\frac{\gamma-1}{4}, c_2=\frac{(\gamma+1)^2}{\gamma-1}$. In addition, if $0<\inf_k\alpha_k\le\sup_k\alpha_k<\infty$ (not necessary $\alpha_k=\frac{1}{\gamma L_k}$), \eqref{eq:dec2-0} holds with positive $c_1$ and $c_2$, and the extrapolation weights satisfy $\tilde{\omega}_i^j\le\delta\sqrt{(c_1\tilde{L}_i^{j-1})/(c_2\tilde{L}_i^j)},\forall i,j$ for some $\delta<1$, then all our convergence results below remain valid. 

Note that $d_i^{k}=d_i^{k-1}+1$ for $i=b_k$ and $d_i^{k}=d_i^{k-1}, \forall i\neq b_k$. Adopting the convention that $\sum_{j=p}^q a_j=0$ when $q<p$, we can write \eqref{eq:dec2-1} into
\begin{equation}\label{eq:diff}
F(\bfx^{k-1})-F(\bfx^k)\ge\sum_{i=1}^s\sum_{j=d_i^{k-1}+1}^{d_i^k}\frac{1}{4}\left(\tilde{L}_i^j\|\tilde{\bfx}_i^{j-1}-\tilde{\bfx}_i^j\|^2-\tilde{L}_i^{j-1}\delta^2\|\tilde{\bfx}_i^{j-2}-\tilde{\bfx}_i^{j-1}\|^2\right),
\end{equation}
which will be used in our subsequent convergence analysis.
\end{remark}
}


\begin{remark}\label{rm:multiconv}
If $f$ is block multi-convex, i.e., it is convex with respect to each block of variables while keeping the remaining variables fixed, and $r_i$ is convex for all $i$, then taking $\alpha_k=\frac{1}{L_k}$, we have \eqref{eq:dec2-0} holds with $c_1=\frac{1}{2}$ and $c_2=\frac{1}{2}$; see the proof in Appendix \ref{app:proof-lem}.
In this case, we can take $\tilde{\omega}_i^j\le \delta\sqrt{\tilde{L}_i^{j-1}/\tilde{L}_i^j},\,\forall i,j$ for some $\delta<1$, and all our convergence results can be shown through the same arguments.
\end{remark}


\subsection{Subsequence convergence}
Using Lemma \ref{lem:dec0}, we can have the following result, through which we show subsequence convergence of Algorithm \ref{alg:ebpg}.
\begin{proposition}[Square summable]\label{prop:sqsum}
Let $\{\bfx^k\}_{k\ge1}$ be generated from Algorithm \ref{alg:ebpg} with $\alpha_k$ and $\omega_k$ taken from \eqref{eq:alpha-omega}. 
We have
\begin{equation}\label{eq:sqsum}
\sum_{k=1}^\infty\|{\vx}^{k-1}-{\vx}^k\|^2<\infty.
\end{equation}
\end{proposition}

\begin{theorem}[Subsequence convergence]\label{thm:subseq}
Under Assumptions \ref{assump1} through \ref{assump3}, let $\{\bfx^k\}_{k\ge1}$ be generated from Algorithm \ref{alg:ebpg} with $\alpha_k$ and $\omega_k$ taken from \eqref{eq:alpha-omega}. Then any limit point $\bar{\vx}$ of $\{\bfx^k\}_{k\ge1}$ is a critical point of \eqref{eq:main}. If the subsequence $\{\vx^k\}_{k\in\bar{\cK}}$ converges to $\bar{\vx}$, then
\begin{equation}\label{fun-lim}
\underset{\bar{\cK}\ni k\to\infty}\lim F(\vx^k)=F(\bar{\vx})
\end{equation}
\end{theorem}

\begin{remark}
The existence of finite limit point is guaranteed if $\{\vx^k\}_{k\ge1}$ is bounded, and for some applications, the boundedness of $\{\vx^k\}_{k\ge1}$ can be satisfied by setting appropriate parameters in Algorithm \ref{alg:ebpg}; see examples in section \ref{sec:example}. If $r_i$'s are continuous, \eqref{fun-lim} immediately holds. {Since we only assume lower semi-continuity of $r_i$'s, $F(\vx)$ may not converge to $F(\bar{\vx})$ as $\vx\to\bar{\vx}$, so \eqref{fun-lim} is not obvious.}
\end{remark}

\begin{proof}
Assume $\bar{\bfx}$ is a limit point of $\{\bfx^k\}_{k\ge1}$. Then there exists an index set $\cK$ so that the subsequence $\{\bfx^k\}_{k\in\cK}$ converging to $\bar{\bfx}$. From \eqref{eq:sqsum}, we have $\|\vx^{k-1}-\vx^k\|\to 0$ and thus $\{\bfx^{k+\kappa}\}_{k\in\cK}\to\bar{\bfx}$ for any $\kappa\ge 0$. Define
$$\cK_i=\{k\in\cup_{\kappa=0}^{T-1}(\cK+\kappa): b_k=i\},\ i=1,\ldots,s.$$
Take an arbitrary $i\in\{1,\ldots,s\}$. Note $\cK_i$ is an infinite set according to Assumption \ref{assump3}. Taking another subsequence if necessary,  $L_k$ converges to some $\bar{L}_i$ as $\cK_i\ni k\to\infty$.
Note that since $\alpha_k=\frac{1}{2L_k},\forall k$, for any $k\in \cK_i$,
\begin{equation}\label{eq:up-tau}
\bfx_i^{k}\in\argmin_{\bfx_i}\,\langle\nabla_{\vx_i} f(\bfx_{\neq i}^{k-1},\hat{\bfx}_i^k), \bfx_i-\hat{\bfx}_i^k\rangle+L_k\|\bfx_i-\hat{\bfx}_i^k\|^2+r_i(\bfx_i).
\end{equation}
Note from \eqref{eq:sqsum} and \eqref{eq:extrap} that
$\hat{\vx}_i^k\to \bar{\vx}_i$ as $\cK_i\ni k\to\infty$. Since $f$ is continuously differentiable and $r_i$ is lower semicontinuous, letting $\cK_i\ni k\to\infty$ in \eqref{eq:up-tau} yields
\begin{align*}
r_i(\bar{\vx}_i)\le & \liminf_{\cK_i\ni k\to\infty}\left(\nabla_{\vx_i} f(\bfx_{\neq i}^{k-1},\hat{\bfx}_i^k), \bfx_i^k-\hat{\bfx}_i^k\rangle+L_k\|\bfx_i^k-\hat{\bfx}_i^k\|^2+r_i(\bfx_i^k)\right)\\
\overset{\eqref{eq:up-tau}}\le &\liminf_{\cK_i\ni k\to\infty}\left(\nabla_{\vx_i} f(\bfx_{\neq i}^{k-1},\hat{\bfx}_i^k), \bfx_i-\hat{\bfx}_i^k\rangle+L_k\|\bfx_i-\hat{\bfx}_i^k\|^2+r_i(\bfx_i)\right),\quad\forall \vx_i\in\dom(F)\\
= & \langle\nabla_{\vx_i} f(\bar{\bfx}), \bfx_i-\bar{\bfx}_i\rangle+\bar{L}_i\|\bfx_i-\bar{\bfx}_i\|^2+r_i(\bfx_i),\quad\forall \vx_i\in\dom(F).
\end{align*}
Hence,
$$\bar{\bfx}_i\in\argmin_{\bfx_{i}}\,\langle\nabla_{\vx_i} f(\bar{\bfx}), \bfx_i-\bar{\bfx}_i\rangle+\bar{L}_i\|\bfx_i-\bar{\bfx}_i\|^2+r_i(\bfx_i),$$
and $\bar{\bfx}_i$ satisfies the first-order optimality condition:
\begin{equation}\label{eq:blockstat}\mathbf{0}\in \nabla_{\vx_i} f(\bar{\bfx})+\partial r_i(\bar{\bfx}_i).
\end{equation}
Since \eqref{eq:blockstat} holds for arbitrary $i\in\{1,\ldots,s\}$,  $\bar{\bfx}$ is a critical point of \eqref{eq:main}.

In addition, \eqref{eq:up-tau} implies
$$\langle\nabla_{\vx_i} f(\bfx_{\neq i}^{k-1},\hat{\bfx}_i^k), \bfx_i^k-\hat{\bfx}_i^k\rangle+L_k\|\bfx_i^k-\hat{\bfx}_i^k\|^2+r_i(\bfx_i^k)\le \langle\nabla_{\vx_i} f(\bfx_{\neq i}^{k-1},\hat{\bfx}_i^k), \bar{\bfx}_i-\hat{\bfx}_i^k\rangle+L_k\|\bar{\bfx}_i-\hat{\bfx}_i^k\|^2+r_i(\bar{\bfx}_i).$$
Taking limit superior on both sides of the above inequality over $k\in\cK_i$ gives $\underset{\cK_i\ni k\to\infty}\limsup r_i(\bfx_i^k)\le r_i(\bar{\bfx}_i).$
Since $r_i$ is lower semi-continuous, it holds $\underset{\cK_i\ni k\to\infty}\liminf r_i(\bfx_i^k)\ge r_i(\bar{\bfx}_i)$, and thus
$$\lim_{\cK_i\ni k\to\infty}r_i(\vx^k_i)=r_i(\bar{\bfx}_i),\, i=1,\ldots,s.$$
Noting that $f$ is continuous, we complete the proof.
\hfill
\end{proof}

\subsection{Whole sequence convergence and rate}
In this subsection, we establish the whole sequence convergence and rate of Algorithm \ref{alg:ebpg} by assuming Condition \ref{cond-dec}. 
%
%
{We first show that under mild assumptions, Condition \ref{cond-dec} holds for certain $\omega_k>0$.} 

\begin{proposition}\label{prop-cvx}
Let $i=b_k$. Assume 
$\prox_{\alpha_k r_i}$ is single-valued near $\vx_i^{k-1}-\alpha_k\nabla_{\vx_i}f(\vx^{k-1})$ and
\begin{equation}\label{not-opt}\bfx_i^{k-1}\not\in\underset{\bfx_i}{\argmin}\,\langle \nabla_{\vx_i} f(\bfx^{k-1}), \bfx_i-\bfx^{k-1}_i\rangle+\frac{1}{2\alpha_k}\|\bfx_i-\bfx^{k-1}_i\|^2+r_{i}(\bfx_i),
\end{equation}
namely, progress can still be made by updating the $i$-th block. Then, there is $\bar{\omega}_k>0$ such that for any $\omega_k\in[0, \bar{\omega}_k]$, we have $F(\vx^k)\le F(\vx^{k-1})$.
\end{proposition}

By this proposition, we can find $\omega_k>0$ through backtracking to maintain the monotonicity of $F(\vx^k)$. All the examples in section \ref{sec:example} satisfy the assumptions of Proposition \ref{prop-cvx}. 
The proof of Proposition \ref{prop-cvx} involves the continuity of $\prox_{\alpha_kr_i}$ and is deferred to Appendix \ref{app:prop-cvx}.


Under Condition \ref{cond-dec} and the KL property of $F$ (Definition \ref{def:KL}), we show that the  sequence $\{\bfx^k\}$ converges as long as it has a finite limit point. We first establish a  lemma, which has its own importance and together with the KL property implies Lemma 2.6 of \cite{attouch2013convergence}.

{The result in  Lemma \ref{lem:seq} below is very general because we need to apply it to Algorithm \ref{alg:ebpg} in its general form. To ease  understanding, let us  go over its especial cases. If $s=1$, $n_{1,m}=m$ and $\beta=0$, then \eqref{cond-seq} below with $\alpha_{1,m}=\alpha_m$ and $A_{1,m}=A_m$ reduces to $\alpha_{m+1}A_{m+1}^2\le B_m A_m$, which together with Young's inequality gives $\sqrt{\underline{\alpha}} A_{m+1}\le \frac{\sqrt{\underline{\alpha}}}{2}A_m+\frac{1}{2\sqrt{\underline{\alpha}}}B_m$. Hence, if $\{B_m\}_{m\ge1}$ is summable,  so will be $\{A_m\}_{m\ge1}$. This result can be used to analyze the prox-linear method. The more general case of $s>1$, $n_{i,m}=m,\,\forall i$ and $\beta=0$  applies to the cyclic block prox-linear method. In this case, \eqref{cond-seq} reduces to $\sum_{i=1}^s\alpha_{i,m+1}A_{i,m+1}^2\le B_m\sum_{i=1}^sA_{i,m},$ which together with the Young's inequality implies
\begin{equation}\label{spec-case}
\sqrt{\underline{\alpha}}\sum_{i=1}^sA_{i,m+1}\le
\sqrt{s}\sqrt{\sum_{i=1}^s\alpha_{i,m+1}A_{i,m+1}^2}\le \frac{s\tau}{4} B_m+\frac{1}{\tau}\sum_{i=1}^sA_{i,m},
\end{equation}
where $\tau$ is sufficiently large so that $\frac{1}{\tau}<\sqrt{\underline{\alpha}}$. Less obviously but still, if $\{B_m\}_{m\ge1}$ is summable, so will be $\{A_{i,m}\}_{m\ge1},\,\forall i$. Finally, we will need $\beta>0$ in \eqref{cond-seq} to analyze the accelerated block prox-linear method.}
\begin{lemma}\label{lem:seq}
For nonnegative sequences $\{A_{i,j}\}_{j\ge 0},\{\alpha_{i,j}\}_{j\ge0},\, i=1,\ldots,s$, and $\{B_m\}_{m\ge0}$, if $$0<\underline{\alpha}=\inf_{i,j}\alpha_{i,j}\le\sup_{i,j}\alpha_{i,j}=\overline{\alpha}<\infty,$$ and
\begin{equation}\label{cond-seq}
\sum_{i=1}^s\sum_{j=n_{i,m}+1}^{n_{i,m+1}}\big(\alpha_{i,j}A_{i,j}^2-
\alpha_{i,j-1}\beta^2A_{i,j-1}^2\big)\le B_m\sum_{i=1}^s\sum_{j=n_{i,m-1}+1}^{n_{i,m}}A_{i,j},\, 0\le m\le M,
\end{equation}
where $0\le \beta<1$, and $\{n_{i,m}\}_{m\ge0},\forall i$ are nonnegative integer sequences satisfying: $n_{i,m}\le n_{i,m+1}\le n_{i,m}+ N, \forall i, m$, for some integer $N>0$. Then we have
{\begin{equation}\label{cond-seq0}
\sum_{i=1}^s\sum_{j=n_{i,M_1}+1}^{n_{i,M_2+1}}A_{i,j}\le \frac{4sN}{\underline{\alpha}(1-\beta)^2}\sum_{m=1}^{M_2} B_m+\left(\sqrt{s}+\frac{4\beta\sqrt{\overline{\alpha}sN}}{(1-\beta)\sqrt{\underline{\alpha}}}\right)\sum_{i=1}^s\sum_{j=n_{i,M_1-1}+1}^{n_{i,M_1}}A_{i,j}, \text{ for } 0\le M_1<M_2\le M.
\end{equation}}


 In addition, if $\sum_{m=1}^\infty B_m<\infty$, $\lim_{m\to\infty}n_{i,m}=\infty,\forall i$, and \eqref{cond-seq} holds for all $m$, then we have
\begin{equation}\label{cau-seq}
\sum_{j=1}^\infty A_{i,j}<\infty,\ \forall i.
\end{equation}
\end{lemma}
The proof of this lemma is given in Appendix \ref{app:lem-seq}
\begin{remark}
To apply \eqref{cond-seq} to the convergence analysis of Algorithm \ref{alg:ebpg}, we will use $A_{i,j}$ for $\|\tilde{\vx}_i^{j-1}-\tilde{\vx}_i^j\|$ and relate $\alpha_{i,j}$ to Lipschitz constant $\tilde{L}_i^j$. The second term in the bracket of the left hand side of \eqref{cond-seq} is used to handle the extrapolation used in Algorithm \ref{alg:ebpg}, and we require $\beta<1$ such that the first term can dominate the second one after summation. 

\end{remark}

We also need the following result.
\begin{proposition}\label{prop:asymconvg}
Let $\{\bfx^k\}$ be generated from Algorithm \ref{alg:ebpg}. 
For a specific iteration $k\ge 3T$, assume $\vx^\kappa\in \cB_\rho(\bar{\vx}),\,\kappa=k-3T,k-3T+1,\ldots,k$ for some $\bar{\vx}$ and $\rho>0$. If for each $i$, $\nabla_{\vx_i}f(\vx)$ is Lipschitz continuous with constant $L_G$ within $B_{4\rho}(\bar{\vx})$ with respect to $\vx$, i.e.,
$$\|\nabla_{\vx_i}f(\vy)-\nabla_{\vx_i}f(\vz)\|\le L_G\|\vy-\vz\|,\ \forall \vy,\vz\in B_{4\rho}(\bar{\vx}),$$
then
\begin{equation}\label{dist}
\dist(\vzero,\partial F(\vx^k))\le \big(2(L_G+2L)+sL_G\big)\sum_{i=1}^s\sum_{j=d_i^{k-3T}+1}^{d_i^k}\|\tilde{\vx}_i^{j-1}-\tilde{\vx}_i^j\|.
\end{equation}
\end{proposition}

We are now ready to present and show the whole sequence convergence of Algorithm \ref{alg:ebpg}.
\begin{theorem}[Whole sequence convergence]\label{thm:global-ebpg}
Suppose that Assumptions \ref{assump1} through \ref{assump3} and Condition \ref{cond-dec} hold. Let $\{\bfx^k\}_{k\ge1}$ be generated from Algorithm \ref{alg:ebpg}. 
Assume
\begin{enumerate}
\item $\{\bfx^k\}_{k\ge1}$ has a finite limit point $\bar{\bfx}$; 
\item $F$ satisfies the KL property \eqref{eq:KL} around $\bar{\bfx}$ with parameters $\rho$, $\eta$ and $\theta$. 
\item For each $i$, $\nabla_{\vx_i} f(\bfx)$ is Lipschitz continuous within $B_{4\rho}(\bar{\vx})$ with respect to $\bfx$.
\end{enumerate}
Then
$$\lim_{k\to\infty}\bfx^k=\bar{\bfx}.$$
\end{theorem}

\begin{remark}
Before proving the theorem, let us remark on the conditions 1--3. The condition  1 can be guaranteed if $\{\vx^k\}_{k\ge1}$ has a bounded subsequence. The condition 2 is  satisfied for a broad class of applications as we mentioned in section \ref{sec:KL}. The condition 3 is a weak assumption since it requires the Lipschitz continuity only in a bounded set.
\end{remark}

\begin{proof}
From \eqref{fun-lim} and Condition \eqref{cond-dec}, we have $F(\vx^k)\to F(\bar{\vx})$ as $k\to\infty$.
We consider two cases depending on whether there is an integer $K_0$ such that $F(\vx^{K_0})=F(\bar{\vx})$.

\noindent\textbf{Case 1:} Assume $F(\vx^k)>F(\bar{\vx}),\,\forall k$.

Since $\bar{\bfx}$ is a limit point of $\{\vx^k\}$ and according to \eqref{eq:sqsum}, one can choose a sufficiently large $k_0$ such that the points $\bfx^{k_0+\kappa}, \kappa=0,1,\ldots,3T$ are all sufficiently close to $\bar{\bfx}$ and in $\cB_\rho(\bar{\bfx})$, and also the differences $\|\bfx^{k_0+\kappa}-\bfx^{k_0+\kappa+1}\|, \kappa = 0,1,\ldots,3T$ are sufficiently close to \emph{zero}. In addition, note that $F(\vx^k)\to F(\bar{\vx})$ as $k\to\infty$, and thus both $F(\vx^{3(k_0+1)T})-F(\bar{\vx})$ and $\phi(F(\vx^{3(k_0+1)T})-F(\bar{\vx}))$ can be sufficiently small. Since $\{\vx^k\}_{k\ge0}$ converges if and only if $\{\vx^k\}_{k\ge k_0}$ converges, without loss of generality, we assume $k_0=0$, which is equivalent to setting $\vx^{k_0}$ as a new starting point, and thus we assume
\begin{subequations}\label{suff-close}
\begin{align}
&F(\vx^{3T})-F(\bar{\vx}) < \eta,\label{suff-close1}\\
&C\phi\big(F(\vx^{3T})-F(\bar{\vx})\big)+C\sum_{i=1}^s\sum_{j=1}^{d_i^{3T}}\|\tilde{\vx}_i^{j-1}-\tilde{\vx}_i^j\|+\sum_{i=1}^s\|\tilde{\vx}_i^{d_i^{3T}}-\bar{\vx}_i\|\le\rho, \label{suff-close2}
\end{align}
\end{subequations}
where
{\begin{equation}\label{eq-C}
C=\frac{48sT\big(2(L_G+2L)+sL_G\big)}{\ell(1-\delta)^2}\ge \sqrt{s}+\frac{4\delta\sqrt{3sTL}}{(1-\delta)\sqrt{\ell}}.
\end{equation}
}

Assume that $\vx^{3mT}\in\cB_\rho(\bar{\bfx})$ and $F(\vx^{3mT})<F(\bar{\vx})+\eta,\, m = 0,\ldots, M$ for some $M\ge 1$. Note that from \eqref{suff-close}, we can take $M=1$. Letting $k=3mT$ in \eqref{dist} and
using KL inequality \eqref{eq:KL}, we have
\begin{equation}\label{dist1}
\phi'(F(\vx^{3mT})-F(\bar{\vx}))\left(\big(2(L_G+2L)+sL_G\big)\sum_{i=1}^s\sum_{j=d_i^{3(m-1)T}+1}^{d_i^{3mT}}\|\tilde{\vx}_i^{j-1}-\tilde{\vx}_i^j\|\right)\ge 1,
\end{equation}
where $L_G$ is a uniform Lipschitz constant of $\nabla_{\vx_i}f(\vx),\forall i$ within $\cB_{4\rho}(\bar{\vx})$.
In addition, it follows from \eqref{eq:diff} that
\begin{equation}\label{dec3}
F(\vx^{3mT})-F(\vx^{3(m+1)T})\ge \sum_{i=1}^s\sum_{j=d_i^{3m T}+1}^{d_i^{3(m+1)T}}\left(\frac{\tilde{L}_i^j}{4}\|\tilde{\bfx}_i^{j-1}-\tilde{\bfx}_i^j\|^2-\frac{\tilde{L}_i^{j-1}\delta^2}{4}\|\tilde{\bfx}_i^{j-2}-\tilde{\bfx}_i^{j-1}\|^2\right).
\end{equation}

Let $\phi_m=\phi(F(\bfx^{3mT})-F(\bar{\bfx}))$. Note that $$\phi_m-\phi_{m+1}\ge \phi'(F(\bfx^{3mT})-F(\bar{\bfx}))[F(\bfx^{3mT})-F(\bfx^{3(m+1)T})].$$
Combining \eqref{dist1} and \eqref{dec3} with the above inequality and letting $\tilde{C}=2(L_G+2L)+sL_G$ give
\begin{equation}\label{cond-seq-2}
\sum_{i=1}^s\sum_{j=d_i^{3m T}+1}^{d_i^{3(m+1)T}}\left(\frac{\tilde{L}_i^j}{4}\|\tilde{\bfx}_i^{j-1}-\tilde{\bfx}_i^j\|^2-\frac{\tilde{L}_i^{j-1}\delta^2}{4}\|\tilde{\bfx}_i^{j-2}-\tilde{\bfx}_i^{j-1}\|^2\right)
\le \tilde{C}(\phi_m-\phi_{m+1})\sum_{i=1}^s\sum_{j=d_i^{3(m-1)T}+1}^{d_i^{3mT}}\|\tilde{\vx}_i^{j-1}-\tilde{\vx}_i^j\|.
\end{equation}
Letting $A_{i,j}=\|\tilde{\vx}_i^{j-1}-\tilde{\vx}_i^j\|, \alpha_{i,j}=\tilde{L}_i^j/4$, $n_{i,m}=d_i^{3mT}$, $B_m=\tilde{C}(\phi_m-\phi_{m+1})$, and $\beta=\delta$ in Lemma \ref{lem:seq}, we note $d_i^{3m(T+1)}-d_i^{3mT}\le 3T$ and have from \eqref{cond-seq0} that for any intergers $N$ and $M$,
\begin{align}\label{app-key3}
\sum_{i=1}^s\sum_{j=d_i^{3NT}+1}^{d_i^{3(M+1)T}}\|\tilde{\vx}_i^{j-1}-\tilde{\vx}_i^j\|\le C\phi_N+C\sum_{i=1}^s\sum_{j=d_i^{3(N-1)T}+1}^{d_i^{3NT}}\|\tilde{\vx}_i^{j-1}-\tilde{\vx}_i^j\|,
\end{align}
where $C$ is given in \eqref{eq-C}.
Letting $N=1$ in the above inequality, we have
\begin{align*}
\|\vx^{3(M+1)T}-\bar{\vx}\|\le & \sum_{i=1}^s\|\tilde{\vx}_i^{d_i^{3(M+1)T}}-\bar{\vx}_i\|\cr
\le &\sum_{i=1}^s\left(\sum_{j=d_i^{3T}+1}^{d_i^{3(M+1)T}}\|\tilde{\vx}_i^{j-1}-\tilde{\vx}_i^j\|+\|\tilde{\vx}_i^{d_i^{3T}}-\bar{\vx}_i\|\right)\cr
\le & C\phi_1+C\sum_{i=1}^s\sum_{j=1}^{d_i^{3T}}\|\tilde{\vx}_i^{j-1}-\tilde{\vx}_i^j\|+\sum_{i=1}^s\|\tilde{\vx}_i^{d_i^{3T}}-\bar{\vx}_i\|\overset{\eqref{suff-close2}}\le\rho.
\end{align*}

Hence, $\bfx^{3(M+1)T}\in \cB_\rho(\bar{\bfx})$. In addition $F(\bfx^{3(M+1)T})\le F(\vx^{3MT})<F(\bar{\vx})+\eta$. By induction, $\bfx^{3mT}\in \cB_\rho(\bar{\bfx}),\forall m$,  and \eqref{app-key3} holds for all $M$. Using Lemma \ref{lem:seq} again, we have that $\{\tilde{\bfx}_i^j\}$ is a Cauchy sequence for all $i$ and thus converges, and $\{\vx^k\}$ also converges. Since $\bar{\bfx}$ is a limit point of $\{\bfx^k\}$, we have $\bfx^k\to\bar{\bfx}$, as $k\to\infty$.

\noindent\textbf{Case 2:} Assume $F(\vx^{K_0})=F(\bar{\vx})$ for a certain integer $K_0$.

Since $F(\vx^k)$ is nonincreasingly convergent to $F(\bar{\vx})$, we have $F(\vx^k)=F(\bar{\vx}),\,\forall k\ge K_0$. Take $M_0$ such that $3M_0T\ge K_0$. Then $F(\vx^{3mT})=F(\vx^{3(m+1)T})=F(\bar{\vx}),\,\forall m\ge M_0$. Summing up \eqref{dec3} from $m=M\ge M_0$ gives
\begin{align}\label{eq:summ-M}
0\ge & \sum_{m=M}^\infty\sum_{i=1}^s\sum_{j=d_i^{3m T}+1}^{d_i^{3(m+1)T}}\left(\frac{\tilde{L}_i^j}{4}\|\tilde{\bfx}_i^{j-1}-\tilde{\bfx}_i^j\|^2-\frac{\tilde{L}_i^{j-1}\delta^2}{4}\|\tilde{\bfx}_i^{j-2}-\tilde{\bfx}_i^{j-1}\|^2\right)\cr
=& \sum_{m=M}^\infty\sum_{i=1}^s\sum_{j=d_i^{3m T}+1}^{d_i^{3(m+1)T}}\frac{\tilde{L}_i^j(1-\delta^2)}{4}\|\tilde{\bfx}_i^{j-1}-\tilde{\bfx}_i^j\|^2 - \sum_{i=1}^s\sum_{j=d_i^{3m T}}\frac{\tilde{L}_i^j\delta^2}{4}\|\tilde{\bfx}_i^{j-1}-\tilde{\bfx}_i^j\|^2.
\end{align}
Let $$a_m=\sum_{i=1}^s\sum_{j=d_i^{3m T}+1}^{d_i^{3(m+1)T}}\|\tilde{\bfx}_i^{j-1}-\tilde{\bfx}_i^j\|^2, \qquad S_M=\sum_{m=M}^\infty a_m.$$
Noting $\ell\le\tilde{L}_i^j\le L$, we have from \eqref{eq:summ-M} that
$\ell(1-\delta^2)S_{M+1}\le L\delta^2(S_M-S_{M+1})$ and thus
$$S_{M}\le \gamma^{M-M_0} S_{M_0},\,\forall M\ge M_0,$$
where $\gamma=\frac{L\delta^2}{L\delta^2+\ell(1-\delta^2)}<1$. By the Cauchy-Schwarz inequality and noting that $a_m$ is the summation of at most $3T$ nonzero terms, we have
\begin{align}\label{eq:exp-term}
\sum_{i=1}^s\sum_{j=d_i^{3m T}+1}^{d_i^{3(m+1)T}}\|\tilde{\bfx}_i^{j-1}-\tilde{\bfx}_i^j\|\le \sqrt{3T}\sqrt{a_m}\le \sqrt{3T}\sqrt{S_m}\le \sqrt{3T}\gamma^{\frac{m-M_0}{2}}S_{M_0},\,\forall m\ge M_0.
\end{align}
Since $\gamma<1$, \eqref{eq:exp-term} implies
$$\sum_{m=M_0}^\infty\sum_{i=1}^s\sum_{j=d_i^{3m T}+1}^{d_i^{3(m+1)T}}\|\tilde{\bfx}_i^{j-1}-\tilde{\bfx}_i^j\|\le\frac{\sqrt{3T}S_{M_0}}{1-\sqrt{\gamma}}<\infty,$$
and thus $\vx^k$ converges to the limit point $\bar{\vx}$. This completes the proof.
\hfill\end{proof}

In addition, we can show convergence rate of Algorithm \ref{alg:ebpg} through the following lemma. 

\begin{lemma}\label{lem:rate-seq1}
For nonnegative sequence $\{A_k\}_{k=1}^\infty$, if $A_k\le A_{k-1}\le 1,\, \forall k\ge K$ for some integer $K$, and there are positive constants $\alpha,\beta$ and $\gamma$ such that
\begin{equation}\label{cond-seq1}
A_k\le \alpha(A_{k-1}-A_k)^\gamma+\beta(A_{k-1}-A_k),\,\forall k,
\end{equation}
we have
\begin{enumerate}
\item If $\gamma\ge 1$, then $A_k\le \big(\frac{\alpha+\beta}{1+\alpha+\beta}\big)^{k-K}A_K,\,\forall k\ge K$;
\item If $0<\gamma<1$, then $A_k\le \nu(k-K)^{-\frac{\gamma}{1-\gamma}},\,\forall k\ge K,$ for some positive constant $\nu$.
\end{enumerate}
\end{lemma}

\begin{theorem}[Convergence rate]\label{thm:rate}
Under the assumptions of Theorem \ref{thm:global-ebpg}, we have:
\begin{enumerate}
\item If $\theta\in[0,\frac{1}{2}]$,  $\|\bfx^k-\bar{\bfx}\|\le C\alpha^k, \forall k$, for a certain $C>0, ~\alpha\in[0,1)$;
\item If $\theta\in(\frac{1}{2},1)$,  $\|\bfx^k-\bar{\bfx}\|\le Ck^{-(1-\theta)/(2\theta-1)}, \forall k$, for a certain $C>0$.
\end{enumerate}
\end{theorem}

\begin{proof}
When $\theta=0$, then $\phi'(a) = c,\forall a$, and there must be a sufficiently large integer $k_0$ such that $F(\vx^{k_0})=F(\bar{\vx})$, and thus $F(\vx^k)=F(\bar{\vx}), \forall k\ge k_0$, by noting $F(\vx^{k-1})\ge F(\vx^k)$ and $\lim_{k\to\infty}F(\vx^k)=F(\bar{\vx})$. Otherwise $F(\vx^k)>F(\bar{\vx}),\forall k$. Then from the KL inequality \eqref{eq:KL}, it holds that $c\cdot\dist(\vzero,\partial F(\vx^k))\ge 1,$ for all $\vx^k\in \cB_\rho(\bar{\vx})$, which is impossible since $\dist(\vzero,\partial F(\vx^{3mT}))\to 0$ as $m\to\infty$ from \eqref{dist}.

For $k>k_0$, since $F(\vx^{k-1})=F(\vx^k)$, and noting that in \eqref{eq:diff} all terms but one are zero under the summation over $i$, 
we have
$$\sum_{i=1}^s\sum_{j=d_i^{k-1}+1}^{d_i^k}\sqrt{\tilde{L}_i^{j-1}}
\delta\|\tilde{\bfx}_i^{j-2}-\tilde{\bfx}_i^{j-1}\|\ge \sum_{i=1}^s\sum_{j=d_i^{k-1}+1}^{d_i^k}\sqrt{\tilde{L}_i^j}
\|\tilde{\bfx}_i^{j-1}-\tilde{\bfx}_i^j\|.$$
Summing the above inequality over $k$ from $m>k_0$ to $\infty$ and using $\ell\le \tilde{L}_i^j\le L,\forall i,j$, we have
\begin{equation}\label{sumseq1}
\sqrt{L}\delta\sum_{i=1}^s\|\tilde{\vx}_i^{d_i^{m-1}-1}- \tilde{\vx}_i^{d_i^{m-1}}\|\ge \sqrt{\ell}(1-\delta)\sum_{i=1}^s\sum_{j=d_i^{m-1}+1}^\infty\|\tilde{\bfx}_i^{j-1}-\tilde{\bfx}_i^j\|,\ \forall m>k_0.
\end{equation}
Let
$$B_m=\sum_{i=1}^s\sum_{j=d_i^{m-1}+1}^\infty\|\tilde{\bfx}_i^{j-1}-\tilde{\bfx}_i^j\|.$$
Then from Assumption \ref{assump3}, we have
$$B_{m-T}-B_m=\sum_{i=1}^s\sum_{j=d_i^{m-T-1}+1}^{d_i^{m-1}}\|\tilde{\bfx}_i^{j-1}-\tilde{\bfx}_i^j\|\ge\sum_{i=1}^s\|\tilde{\vx}_i^{d_i^{m-1}-1}- \tilde{\vx}_i^{d_i^{m-1}}\|.$$
which together with \eqref{sumseq1} gives
$B_m\le \frac{\sqrt{L}\delta}{\sqrt{\ell}(1-\delta)}(B_{m-T}-B_m).$
Hence, $$B_{mT}\le\left(\frac{\sqrt{L}\delta}{\sqrt{L}\delta+\sqrt{\ell}(1-\delta)}\right)B_{(m-1)T}\le \left(\frac{\sqrt{L}\delta}{\sqrt{L}\delta+\sqrt{\ell}(1-\delta)}\right)^{m-\ell_0} B_{\ell_0 T},$$
where $\ell_0=\min\{\ell: \ell T\ge k_0\}$. 
Letting $\alpha=\big(\frac{\sqrt{L}\delta}{\sqrt{L}\delta+\sqrt{\ell}(1-\delta)}\big)^{1/T}$, we have
\begin{equation}\label{exp-convg}B_{mT}\le \alpha^{mT}\big(\alpha^{-\ell_0 T} B_{\ell_0 T}\big).
\end{equation}
Note $\|\vx^{m-1}-\bar{\vx}\|\le B_m$. Hence, choosing a sufficiently large $C>0$ gives the result in item 1 for $\theta=0$.

When $0<\theta<1$, if for some $k_0$, $F(\vx^{k_0})=F(\bar{\vx})$, we have \eqref{exp-convg} by the same arguments as above and thus obtain linear convergence. Below we assume $F(\vx^k)>F(\bar{\vx}),\,\forall k$. Let
$$A_m=\sum_{i=1}^s\sum_{j=d_i^{3mT}+1}^\infty\|\tilde{\vx}_i^{j-1}-\tilde{\vx}_i^j\|,$$
and thus
$$A_{m-1}-A_m=\sum_{i=1}^s\sum_{j=d_i^{3(m-1)T}+1}^{d_i^{3mT}}\|\tilde{\vx}_i^{j-1}-\tilde{\vx}_i^j\|.$$
From \eqref{dist1}, it holds that
$$c(1-\theta)\big(F(\vx^{3mT})-F(\bar{\vx})\big)^{-\theta}\ge \big((2(L_G+2L)+sL_G)(A_{m-1}-A_m)\big)^{-1},$$
which implies
\begin{equation}\label{phim}\phi_m=c\big(F(\vx^{3mT})-F(\bar{\vx})\big)^{1-\theta}\le c\left(c(1-\theta)(2(L_G+2L)+sL_G)(A_{m-1}-A_m)\right)^{\frac{1-\theta}{\theta}}.
\end{equation}
In addition, letting $N=m$ in \eqref{app-key3}, we have
\begin{align*}
\sum_{i=1}^s\sum_{j=d_i^{3mT}+1}^{d_i^{3(M+1)T}}\|\tilde{\vx}_i^{j-1}-\tilde{\vx}_i^j\|\le C\phi_m+C\sum_{i=1}^s\sum_{j=d_i^{3(m-1)T}+1}^{d_i^{3mT}}\|\tilde{\vx}_i^{j-1}-\tilde{\vx}_i^j\|,
\end{align*}
where $C$ is the same as that in \eqref{app-key3}.
Letting $M\to\infty$ in the above inequality, we have
$$A_m\le C_1\phi_m+C_1(A_{m-1}-A_m)\le C_1c\left(c(1-\theta)(2(L_G+2L)+sL_G)(A_{m-1}-A_m)\right)^{\frac{1-\theta}{\theta}}+C_1(A_{m-1}-A_m),$$
where the second inequality is from \eqref{phim}. Since $A_{m-1}-A_m\le 1$ as $m$ is sufficiently large and $\|\vx^m-\bar{\vx}\|\le A_{\lfloor \frac{m}{3T}\rfloor}$, the results in item 2 for $\theta\in(0,\frac{1}{2}]$ and item 3 now immediately follow from Lemma \ref{lem:rate-seq1}.
\hfill\end{proof}

Before closing this section, let us make some comparison to the recent work \cite{xu2013block}. The whole sequence convergence and rate results in this paper are the same as those in \cite{xu2013block}. However, the results here cover more applications. We do not impose any convexity assumption on \eqref{eq:main} while \cite{xu2013block} requires $f$ to be block-wise convex and every $r_i$ to be convex. In addition, the results in \cite{xu2013block} only apply to cyclic block prox-linear method. Empirically, a different block-update order can give better performance. As demonstrated in \cite{chang2008coordinate}, random shuffling can often improve the efficiency of the coordinate descent method for linear support vector machine, and \cite{xu2015spntd} shows that for the Tucker tensor decomposition (see \eqref{ntd}), updating the core tensor more frequently can be better than cyclicly updating the core tensor and factor matrices. 

\section{Applications and numerical results}\label{sec:example}
In this section, we give some specific examples of \eqref{eq:main} and show the whole sequence convergence of some existing algorithms. In addition, we demonstrate that maintaining the nonincreasing monotonicity of the objective value can improve the convergence of accelerated gradient method and that updating variables in a random order can improve the performance of Algorithm \ref{alg:ebpg} over that in the cyclic order.

\subsection{FISTA with backtracking extrapolation} FISTA \cite{BeckTeboulle2009} is an accelerated proximal gradient method for solving composite convex problems. It is a special case of Algorithm \ref{alg:ebpg} with $s=1$ and specific $\omega_k$'s. For the readers' convenience, we present the method in Algorithm \ref{alg:fista}, where both $f$ and $g$ are convex functions, and $L_f$ is the Lipschitz constant of $\nabla f(\vx)$. The algorithm reaches the optimal order of convergence rate among first-order methods, but in general, it does not guarantee monotonicity of the objective values. A restarting scheme is studied in \cite{o2013adaptive} that restarts FISTA from $\vx^k$ whenever $F(\vx^{k+1})>F(\vx^k)$ occurs\footnote{Another restarting option is tested based on gradient information}. It is demonstrated that the restarting FISTA can significantly outperform the original one. In this subsection, we show that FISTA with backtracking extrapolation weight can do even better than the restarting one.

\begin{algorithm}[ht]\label{alg:fista}\caption{Fast iterative shrinkage-thresholding algorithm (FISTA)}
\DontPrintSemicolon
\textbf{Goal:} to solve convex problem $\min_\vx F(\vx)= f(\vx)+g(\vx)$\;
\textbf{Initialization:} set $\vx^0=\vx^1$, $t_1=1$, and $\omega_1=0$\;
\For{$k=1,2,\ldots$}{
Let $\hat{\vx}^k=\vx^k+\omega_k(\vx^k-\vx^{k-1})$\;
Update $\vx^{k+1}=\argmin_\vx \langle \nabla f(\hat{\vx}^k),\vx-\hat{\vx}^k\rangle+\frac{L_f}{2}\|\vx-\hat{\vx}^k\|^2+g(\vx)$\;
Set $t_{k+1}=\frac{1+\sqrt{1+4t_k^2}}{2}$ and $\omega_{k+1}=\frac{t_k-1}{t_{k+1}}$
}
\end{algorithm}

We test the algorithms on solving the following problem
$$\min_\vx \frac{1}{2}\|\vA\vx-\vb\|^2+\lambda\|\vx\|_1,$$
where $\vA\in\RR^{m\times n}$ and $\vb\in\RR^m$ are given. In the test, we set $m=100$, $n=2000$ and $\lambda=1$, and we generate the data in the same way as that in \cite{o2013adaptive}: first generate $\vA$ with all its entries independently following standard normal distribution $\cN(0,1)$, then a sparse vector $\vx$ with only 20 nonzero entries independently following $\cN(0,1)$, and finally let $\vb=\vA\vx+\vy$ with the entries in $\vy$ sampled from $\cN(0,0.1)$. This way ensures the optimal solution is approximately sparse. We set $L_f$ to the spectral norm of $\vA^*\vA$ and the initial point to \emph{zero} vector for all three methods. Figure \ref{fig:lasso} plots their convergence behavior, and it shows that the proposed backtracking scheme on $\omega_k$ can significantly improve the convergence of the algorithm.

\begin{figure}
\centering
\includegraphics[width=0.4\textwidth]{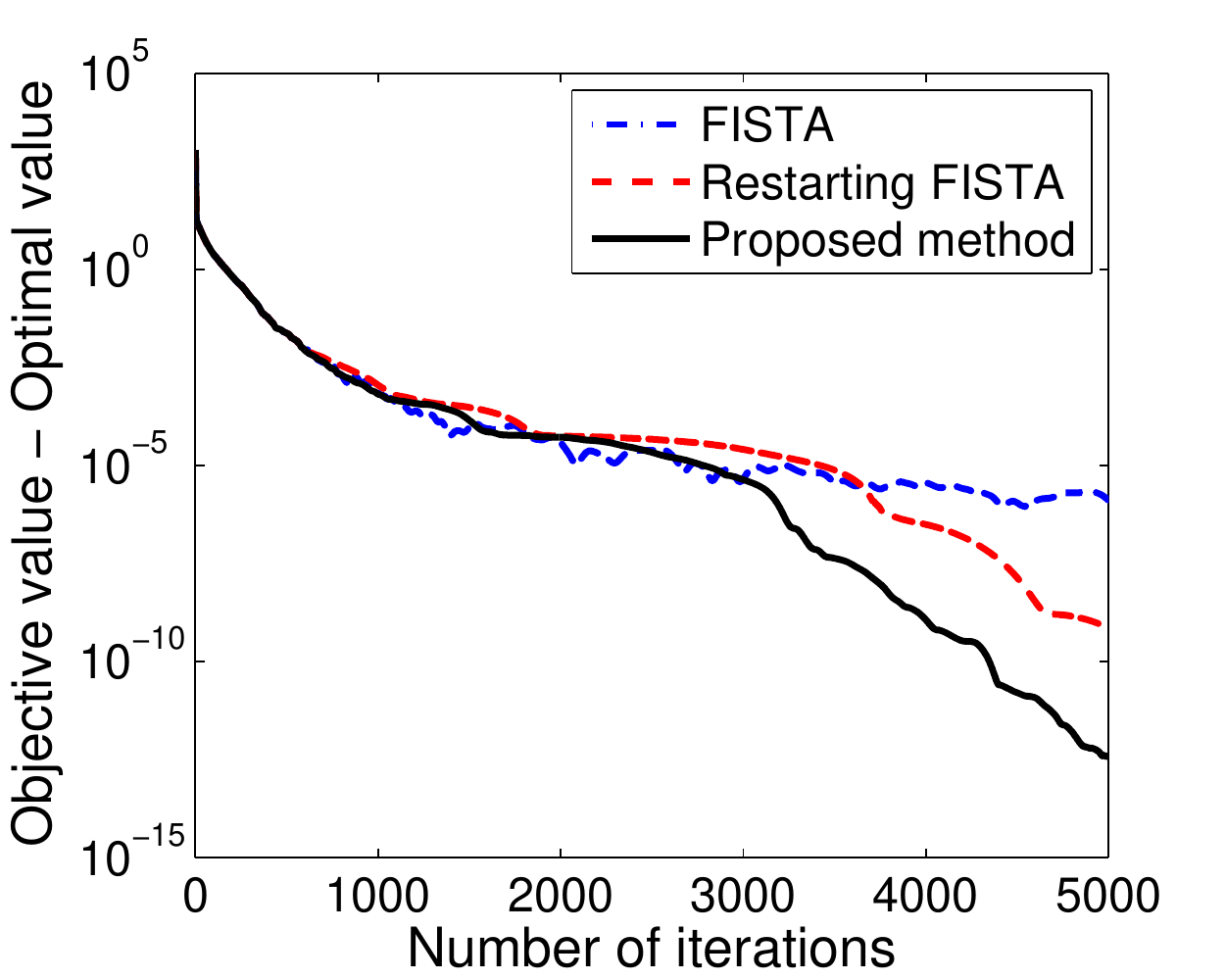}
\caption{Comparison of the FISTA \cite{BeckTeboulle2009}, the restarting FISTA \cite{o2013adaptive}, and the proposed method with backtracking $\omega_k$ to ensure Condition \ref{cond-dec}.}
\label{fig:lasso}
\end{figure}

\subsection{Coordinate descent method for nonconvex regression} As the number of predictors is larger than sample size, variable selection becomes important to keep more important predictors and obtain a more interpretable model, and penalized regression methods are popularly used to achieve variable selection. The work \cite{breheny2011coordinate} considers the linear regression with nonconvex penalties: the minimax concave penalty (MCP) \cite{zhang2010nearly} and the smoothly clipped absolute deviation (SCAD) penalty \cite{fan2001variable}. Specifically, the following model is considered
\begin{equation}\label{noncov}
\min_{\bm{\beta}}\frac{1}{2n}\|\vX\bm{\beta}-\vy\|^2+\sum_{j=1}^p r_{\lambda,\gamma}(\beta_j),
\end{equation}
where $\vy\in\RR^n$ and $\vX\in\RR^{n\times p}$ are standardized such that \begin{equation}\label{stand}
\sum_{i=1}^n y_{i}=0,\ \sum_{i=1}^n x_{ij}=0,\,\forall j,\text{ and }\frac{1}{n}\sum_{i=1}^n x_{ij}^2=1,\, \forall j,
\end{equation}
and MCP is defined as
\begin{equation}\label{mcp}
r_{\lambda,\gamma}(\theta)=\left\{
\begin{array}{ll}
\lambda|\theta|-\frac{\theta^2}{2\gamma},&\text{ if }|\theta|\le \gamma\lambda,\\
\frac{1}{2}\gamma\lambda^2,&\text{ if }|\theta|>\gamma\lambda,
\end{array}
\right.
\end{equation}
and SCAD penalty is defined as
\begin{equation}\label{scad}
r_{\lambda,\gamma}(\theta)=\left\{
\begin{array}{ll}
\lambda|\theta|,&\text{ if }|\theta|\le \lambda,\\
\frac{2\gamma\lambda|\theta|-(\theta^2+\lambda^2)}{2(\gamma-1)},&\text{ if }\lambda<|\theta|\le \gamma\lambda,\\
\frac{\lambda^2(\gamma^2-1)}{2(\gamma-1)},&\text{ if }|\theta|>\gamma\lambda.
\end{array}
\right.
\end{equation}
The cyclic coordinate descent method used in \cite{breheny2011coordinate} performs the update from $j=1$ through $p$
\begin{equation*}
\beta_j^{k+1}=\argmin_{\beta_j}\frac{1}{2n}\|\vX(\bm{\beta}_{<j}^{k+1},\beta_j,\bm{\beta}_{>j}^k)-\vy\|^2+r_{\lambda,\gamma}(\beta_j),
\end{equation*}
which can be equivalently written into the form of \eqref{eq:ebpg} by
\begin{equation}\label{cd-noncov}
\beta_j^{k+1}=\argmin_{\beta_j}\frac{1}{2n}\|\vx_j\|^2(\beta_j-\beta_j^k)^2+\frac{1}{n}\vx_j^\top\big(\vX(\bm{\beta}_{<j}^{k+1},\bm{\beta}_{\ge j}^k)-\vy\big)\beta_j+r_{\lambda,\gamma}(\beta_j).
\end{equation}

Note that the data has been standardized such that $\|\vx_j\|^2=n$. Hence, if $\gamma>1$ in \eqref{mcp} and $\gamma>2$ in \eqref{scad}, it is easy to verify that the objective in \eqref{cd-noncov} is strongly convex, and there is a unique minimizer. From the convergence results of \cite{Tseng-01}, it is concluded in \cite{breheny2011coordinate} that any limit point\footnote{It is stated in  \cite{breheny2011coordinate} that the sequence generated by \eqref{cd-noncov} converges to a coordinate-wise minimizer of \eqref{noncov}. However, the result is obtained directly from \cite{Tseng-01}, which only guarantees subsequence convergence.} of the sequence $\{\bm{\beta}^k\}$ generated by \eqref{cd-noncov} is a coordinate-wise minimizer of \eqref{noncov}. Since $r_{\lambda,\gamma}$ in both \eqref{mcp} and \eqref{scad} is piecewise polynomial and thus semialgebraic, it satisfies the KL property (see Definition \ref{def:KL}). In addition, let $f(\bm{\beta})$ be the objective of \eqref{noncov}. Then
$$f(\bm{\beta}_{< j}^{k+1},\bm{\beta}_{\ge j}^k)-f(\bm{\beta}_{\le j}^{k+1},\bm{\beta}_{>j}^k)\ge \frac{\mu}{2}(\beta_j^{k+1}-\beta_j^{k})^2,$$
where $\mu$ is the strong convexity constant of the objective in \eqref{cd-noncov}. Hence, according to Theorem \ref{thm:global-ebpg} and Remark \ref{rm:large-alpha}, we have the following convergence result.

\begin{theorem}\label{thm:noncov}
Assume $\vX$ is standardized as in \eqref{stand}. Let $\{\bm{\beta}^k\}$ be the sequence generated from \eqref{cd-noncov} or by the following update with random shuffling of coordinates
$$\beta_{\pi^k_j}^{k+1}=\argmin_{\beta_{\pi^k_j}}\frac{1}{2n}\|\vx_{\pi^k_j}\|^2(\beta_{\pi^k_j}-\beta_{\pi^k_j}^k)^2+\frac{1}{n}\vx_{\pi^k_j}^\top\big(\vX(\bm{\beta}_{\pi^k_{<j}}^{k+1},\bm{\beta}_{\pi^k_{\ge j}}^k)-\vy\big)\beta_{\pi^k_j}+r_{\lambda,\gamma}(\beta_{\pi^k_j}),$$
where $(\pi^k_1,\ldots,\pi^k_p)$ is any permutation of $(1,\ldots,p)$, and $r_{\lambda,\gamma}$ is given by either  \eqref{mcp} with $\gamma>1$ or \eqref{scad} with $\gamma>2$. If $\{\bm{\beta}^k\}$ has a finite limit point, then $\bm{\beta}^k$ converges to a coordinate-wise minimizer of \eqref{noncov}.
\end{theorem}

\subsection{Rank-one residue iteration for nonnegative matrix factorization} The nonnegative matrix factorization can be modeled as
\begin{equation}\label{nmf}
\min_{\vX,\vY}\|\vX\vY^\top-\vM||_F^2, \st \vX\in\RR_+^{m\times p},\, \vY\in\RR_+^{n\times p},
\end{equation}
where $\vM\in\RR_+^{m\times n}$ is a given nonnegative matrix, $\RR_+^{m\times p}$ denotes the set of $m\times p$ nonnegative matrices, and $p$ is a user-specified rank. The problem in \eqref{nmf} can be written in the form of \eqref{eq:main} by letting $$f(\vX,\vY)=\frac{1}{2}\|\vX\vY^\top-\vM||_F^2,\quad  r_1(\vX)=\iota_{\RR_+^{m\times p}}(\vX), \quad r_2(\vY)=\iota_{\RR_+^{n\times p}}(\vY).$$

In the literature, most existing algorithms for solving \eqref{nmf} update $\vX$ and $\vY$ alternatingly; see the review paper \cite{kim2014algorithms} and the references therein.
The work \cite{ho2011descent} partitions the variables in a different way: $(\vx_1,\vy_1,\ldots,\vx_p,\vy_p)$, where $\vx_j$ denotes the $j$-th column of $\vX$, and proposes the rank-one residue iteration (RRI) method. It updates the variables cyclically, one column at a time. Specifically, RRI performs the updates cyclically from $i=1$ through $p$,
\begin{subequations}\label{alg:rri}
\begin{align}
\vx_i^{k+1} = &\argmin_{\vx_i\ge 0}\|\vx_i(\vy_i^k)^\top+\vX_{<i}^{k+1}(\vY_{<i}^{k+1})^\top+\vX_{>i}^{k}(\vY_{>i}^{k})^\top-\vM\|_F^2,\\
\vy_i^{k+1} = & \argmin_{\vy_i\ge 0}\|\vx_i^{k+1}(\vy_i)^\top+\vX_{<i}^{k+1}(\vY_{<i}^{k+1})^\top+\vX_{>i}^{k}(\vY_{>i}^{k})^\top-\vM\|_F^2,
\end{align}
\end{subequations}
where $\vX_{>i}^k=(\vx_{i+1}^k,\ldots,\vx_p^k)$. It is a cyclic block minimization method, a special case of \cite{Tseng-01}. The advantage of RRI is that each update in \eqref{alg:rri} has a closed form solution. Both updates in \eqref{alg:rri} can be written in the form of \eqref{eq:ebpg} by noting that they are equivalent to
\begin{subequations}\label{equpdate}
\begin{align}
\vx_i^{k+1}=&\argmin_{\vx_i\ge 0} \frac{1}{2}\|\vy_i^k\|^2\|\vx_i-\vx_i^k\|^2+(\vy_i^k)^\top\big(\vX_{<i}^{k+1}(\vY_{<i}^{k+1})^\top+\vX_{\ge i}^{k}(\vY_{\ge i}^{k})^\top-\vM\big)^\top \vx_i,\label{equpdate-x}\\
\vy_i^{k+1}=&\argmin_{\vy_i\ge 0}\frac{1}{2}\|\vx_i^{k+1}\|^2\|\vy_i-\vy_i^k\|^2+\vy_i^\top\big(\vX_{<i}^{k+1}(\vY_{<i}^{k+1})^\top+\vx_i^{k+1}(\vy_i^k)^\top+\vX_{>i}^{k}(\vY_{>i}^{k})^\top-\vM\big)^\top\vx_i^{k+1}.
\end{align}
\end{subequations}
Since $f(\vX,\vY)+r_1(\vX)+r_2(\vY)$ is semialgebraic and has the KL property, directly from Theorem \ref{thm:global-ebpg}, we have the following whole sequence convergence, which is stronger compared to the subsequence convergence in \cite{ho2011descent}.
\begin{theorem}[Global convergence of RRI]\label{thm:org-rri}
Let $\{(\vX^k,\vY^k)\}_{k=1}^\infty$ be the sequence generated by \eqref{alg:rri} or \eqref{equpdate} from any starting point $(\vX^0,\vY^0)$. If $\{\vx_i^k\}_{i,k}$ and $\{\vy_i^k\}_{i,k}$ are uniformly bounded and away from \emph{zero}, then $(\vX^k,\vY^k)$ converges to a critical point of \eqref{nmf}.
\end{theorem}

However, during the iterations of RRI, it may happen that some columns of $\vX$ and $\vY$ become or approach to zero vector, or some of them blow up, and these cases fail the assumption of Theorem \ref{thm:org-rri}. To tackle with the difficulties, we modify the updates in \eqref{alg:rri} and improve the RRI method as follows.

Our first modification is to require each column of $\vX$ to have unit Euclidean norm; the second modification is to take the Lipschitz constant of $\nabla_{\vx_i}f(\vX_{<i}^{k+1},\vx_i,\vX_{>i}^k,\vY_{<i}^{k+1},\vY_{\ge i}^k)$ to be $L_i^k=\max(L_{\min}, \|\vy_i^k\|^2)$ for some $L_{\min}>0$; the third modification is that at the beginning of the $k$-th cycle, we shuffle the blocks to a permutation $(\pi_1^k,\ldots,\pi_p^k)$. Specifically, we perform the following updates from $i=1$ through $p$,
\begin{subequations}\label{newupdate}
\begin{align}
\vx_{\pi_i^k}^{k+1}=&\argmin_{\vx_{\pi_i^k}\ge 0,\,\|\vx_{\pi_i^k}\|=1} \frac{L_{\pi_i^k}^k}{2}\|\vx_{\pi_i^k}-\vx_{\pi_i^k}^k\|^2+(\vy_{\pi^k_i}^k)^\top\big(\vX_{\pi^k_{<i}}^{k+1}(\vY_{\pi^k_{<i}}^{k+1})^\top+\vX_{\pi^k_{\ge i}}^{k}(\vY_{\pi^k_{\ge i}}^{k})^\top-\vM\big)^\top \vx_{\pi^k_i},\label{newupdate-x}\\
\vy_{\pi^k_i}^{k+1}=&\argmin_{\vy_{\pi^k_i}\ge 0}\frac{1}{2}\|\vy_{\pi^k_i}\|^2+\vy_{\pi^k_i}^\top\big(\vX_{\pi^k_{<i}}^{k+1}(\vY_{\pi^k_{<i}}^{k+1})^\top+\vX_{\pi^k_{>i}}^{k}(\vY_{\pi^k_{>i}}^{k})^\top-\vM\big)^\top\vx_{\pi^k_i}^{k+1}.\label{newupdate-y}
\end{align}
\end{subequations}
Note that if $\pi_i^k=i$ and $L_i^k=\|\vy_i^k\|^2$, the objective in \eqref{newupdate-x} is the same as that in \eqref{equpdate-x}. Both updates in \eqref{newupdate} have closed form solutions; see Appendix \ref{app:newupdate}. Using Theorem \ref{thm:global-ebpg}, we have the following theorem, whose proof is given in Appendix \ref{app:thm-rri}. Compared to the original RRI method, the modified one automatically has bounded sequence and always has the whole sequence convergence.

\begin{theorem}[Whole sequence convergence of modified RRI]\label{thm:rri}
Let $\{(\vX^k,\vY^k)\}_{k=1}^\infty$ be the sequence  generated by \eqref{newupdate} from any starting point $(\vX^0,\vY^0)$. Then $\{\vY^k\}$ is bounded, and $(\vX^k,\vY^k)$ converges to a critical point of \eqref{nmf}.
\end{theorem}

\textbf{Numerical tests.} We tested \eqref{equpdate} and \eqref{newupdate} on randomly generated data and also the Swimmer dataset \cite{donoho2003does}. We set $L_{\min}=0.001$ in the tests and found that \eqref{newupdate} with $\pi_i^k=i,\forall i, k$ produced the same final objective values as those by \eqref{equpdate} on both random data and the Swimmer dataset.  In addition, \eqref{newupdate} with random shuffling performed almost the same as those with $\pi_i^k=i,\forall i$ on randomly generated data. However, random shuffling significantly improved the performance of \eqref{newupdate} on the Swimmer dataset. There are 256 images of resolution $32\times 32$ in the Swimmer dataset, and each image (vectorized to one column of $\vM$) is composed of four limbs and the body. Each limb has four different positions, and all images have the body at the same position; see Figure \ref{fig:swimmer}. Hence, each of these images is a nonnegative combination of 17 images: one with the body and each one of another 16 images with one limb. We set $p=17$ in our test and ran \eqref{equpdate} and \eqref{newupdate} with/without random shuffling to 100 cycles. If the relative error ${\|\vX^{out}(\vY^{out})^\top-\vM\|_F}/{\|\vM\|_F}$ is below $10^{-3}$, we regard the factorization to be successful, where $(\vX^{out},\vY^{out})$ is the output. We ran the three different updates for 50 times independently, and for each run, they were fed with the same randomly generated starting point.  \emph{Both \eqref{equpdate} and \eqref{newupdate} without random shuffling succeed 20 times, and \eqref{newupdate} with random shuffling succeeds 41 times.} Figure \ref{fig:result_swimmer} plots all cases that occur. Every plot is in terms of running time (sec), and during that time, both methods run to 100 cycles. Since \eqref{equpdate} and \eqref{newupdate} without random shuffling give exactly the same results, we only show the results by \eqref{newupdate}. From the figure, we see that \eqref{newupdate} with fixed cyclic order and with random shuffling has similar computational complexity while the latter one can more frequently avoid bad local solutions.
\begin{figure}
\centering
\includegraphics[width=0.7\textwidth]{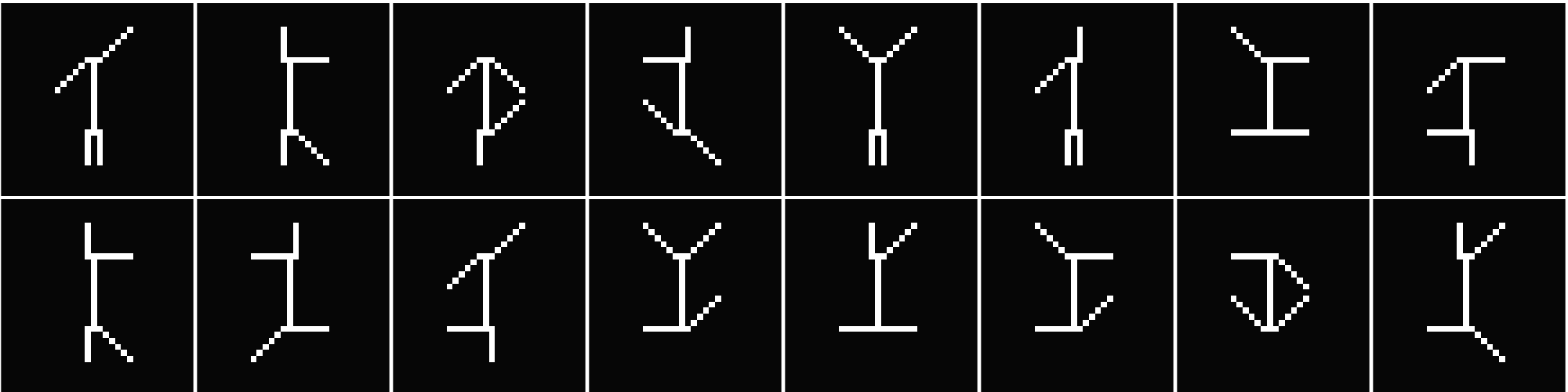}
\caption{Some images in the Swimmer Dataset}\label{fig:swimmer}
\end{figure}

\begin{figure}
\centering
{\footnotesize\begin{tabular}{cccc}
only random succeeds & both succeed & only cyclic succeeds & both fail \\
occurs 25/50 & occurs 16/50 & occurs 4/50 & occurs 5/50 \\
\includegraphics[width=0.23\textwidth]{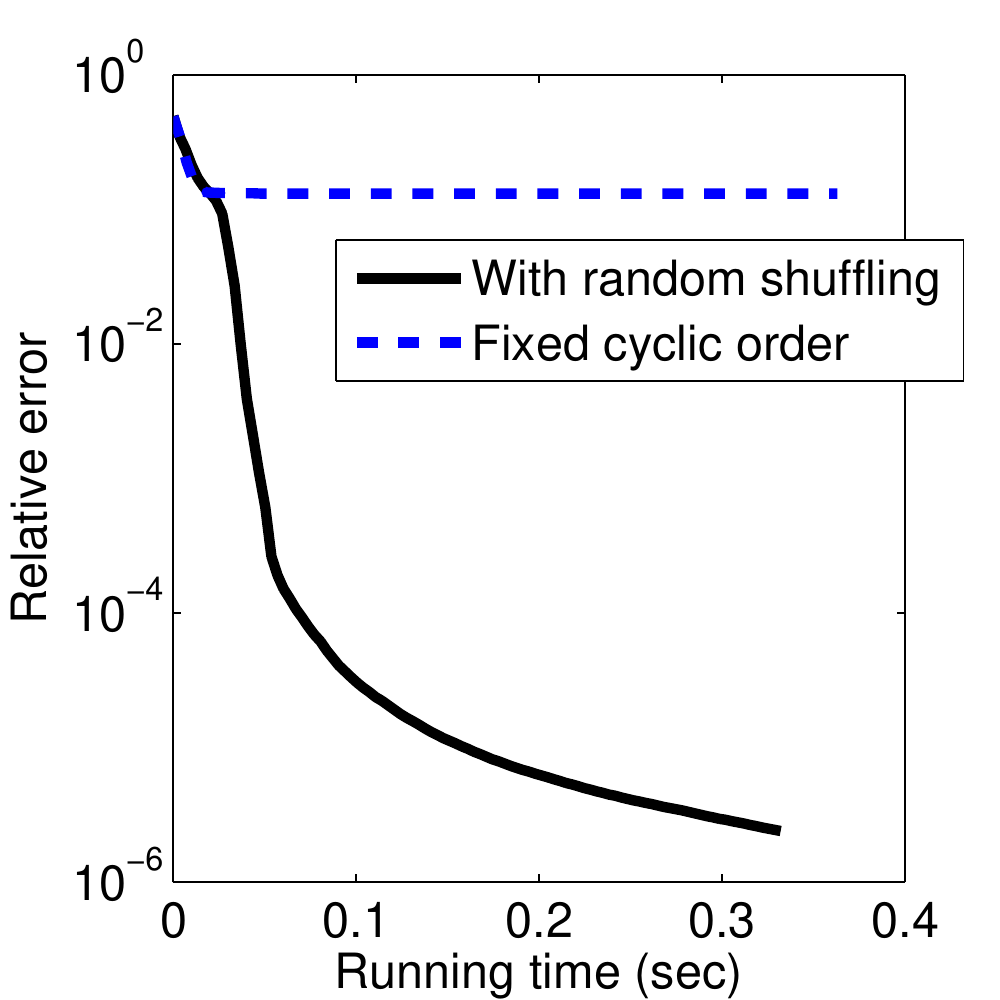}&
\includegraphics[width=0.23\textwidth]{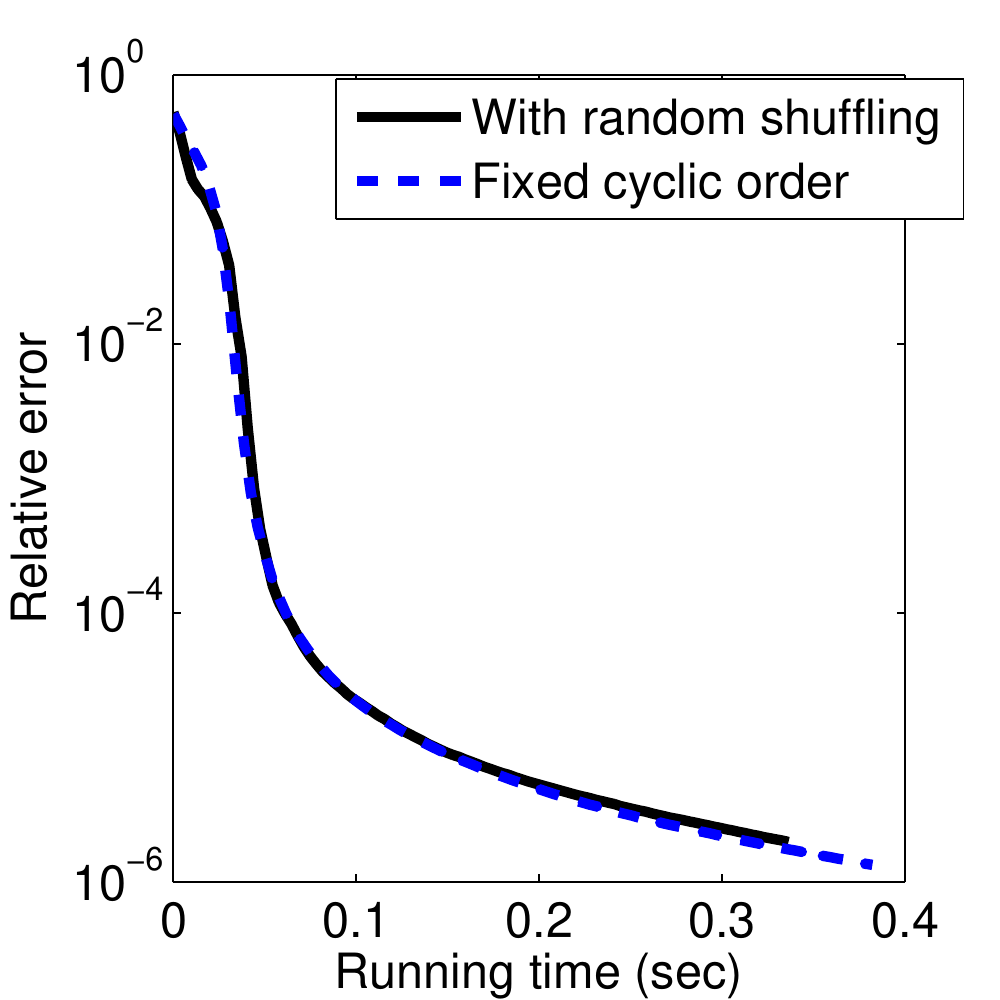}&
\includegraphics[width=0.23\textwidth]{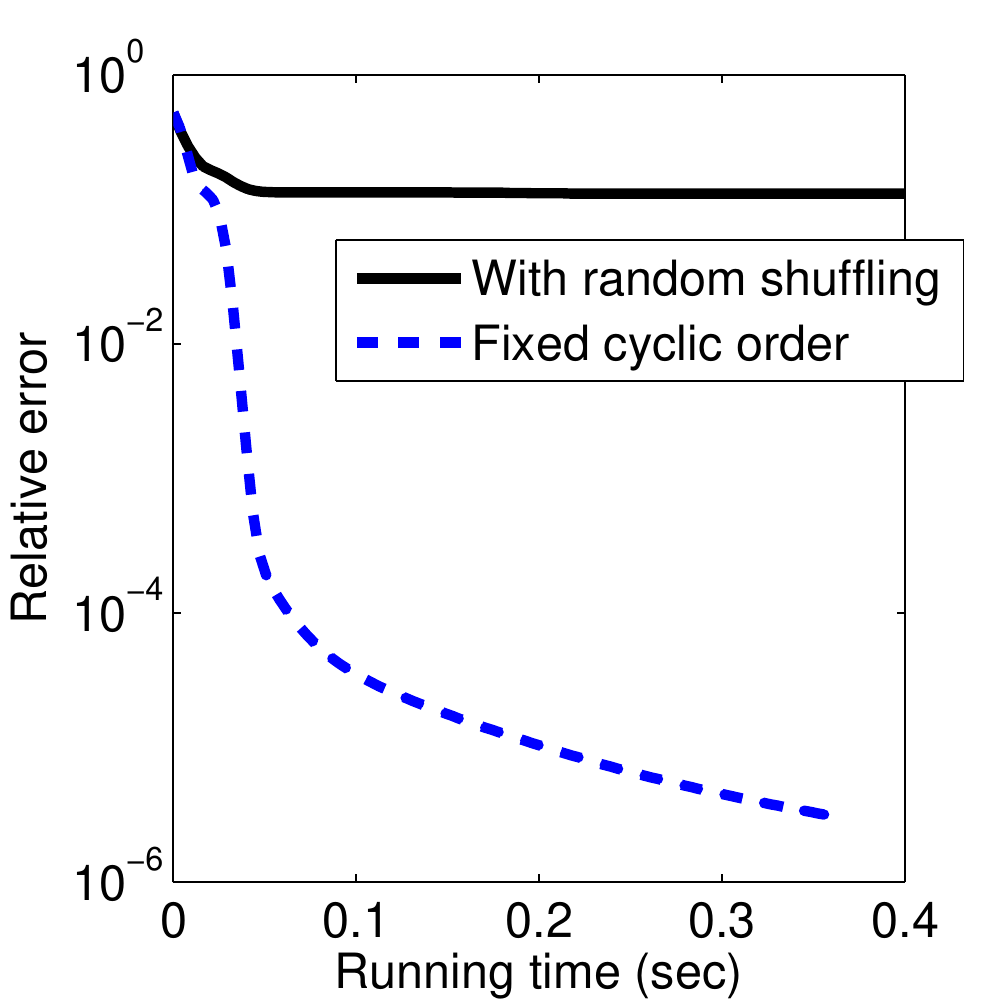}&
\includegraphics[width=0.23\textwidth]{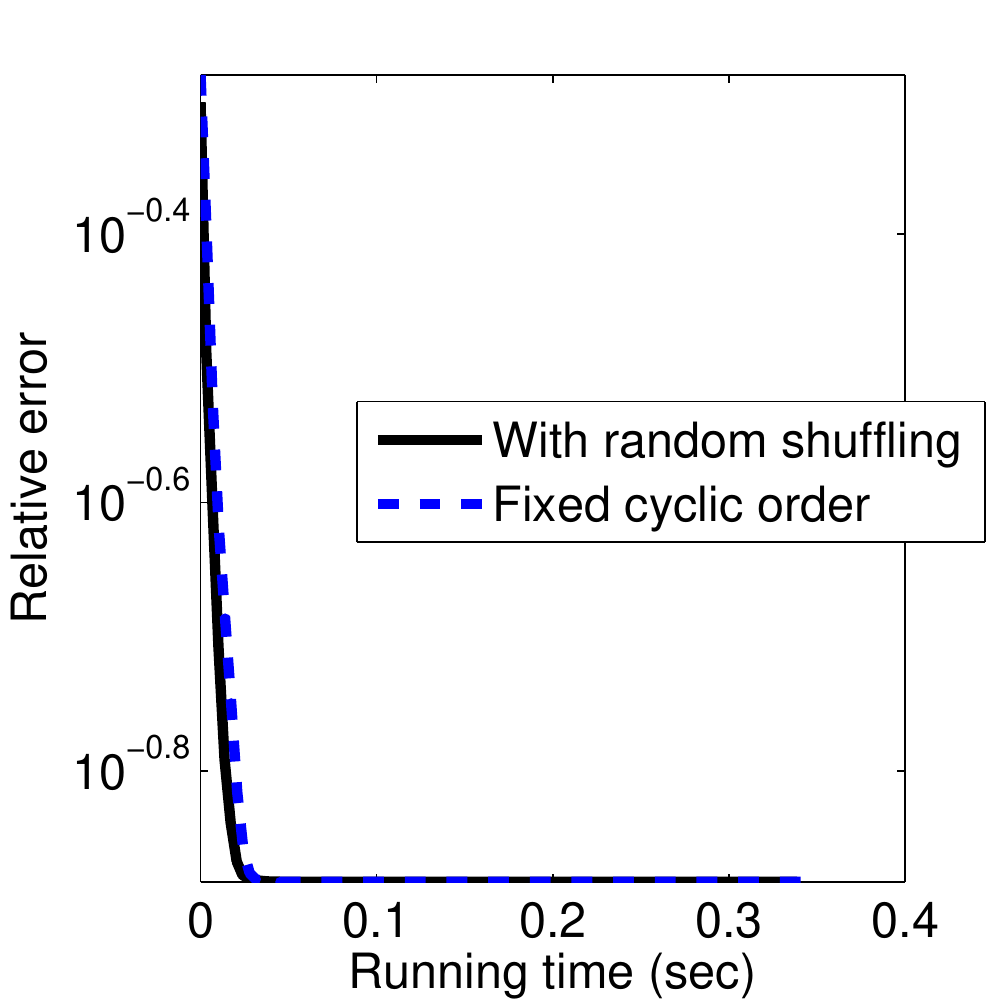}\\
\multicolumn{4}{c}{\includegraphics[width=0.45\textwidth]{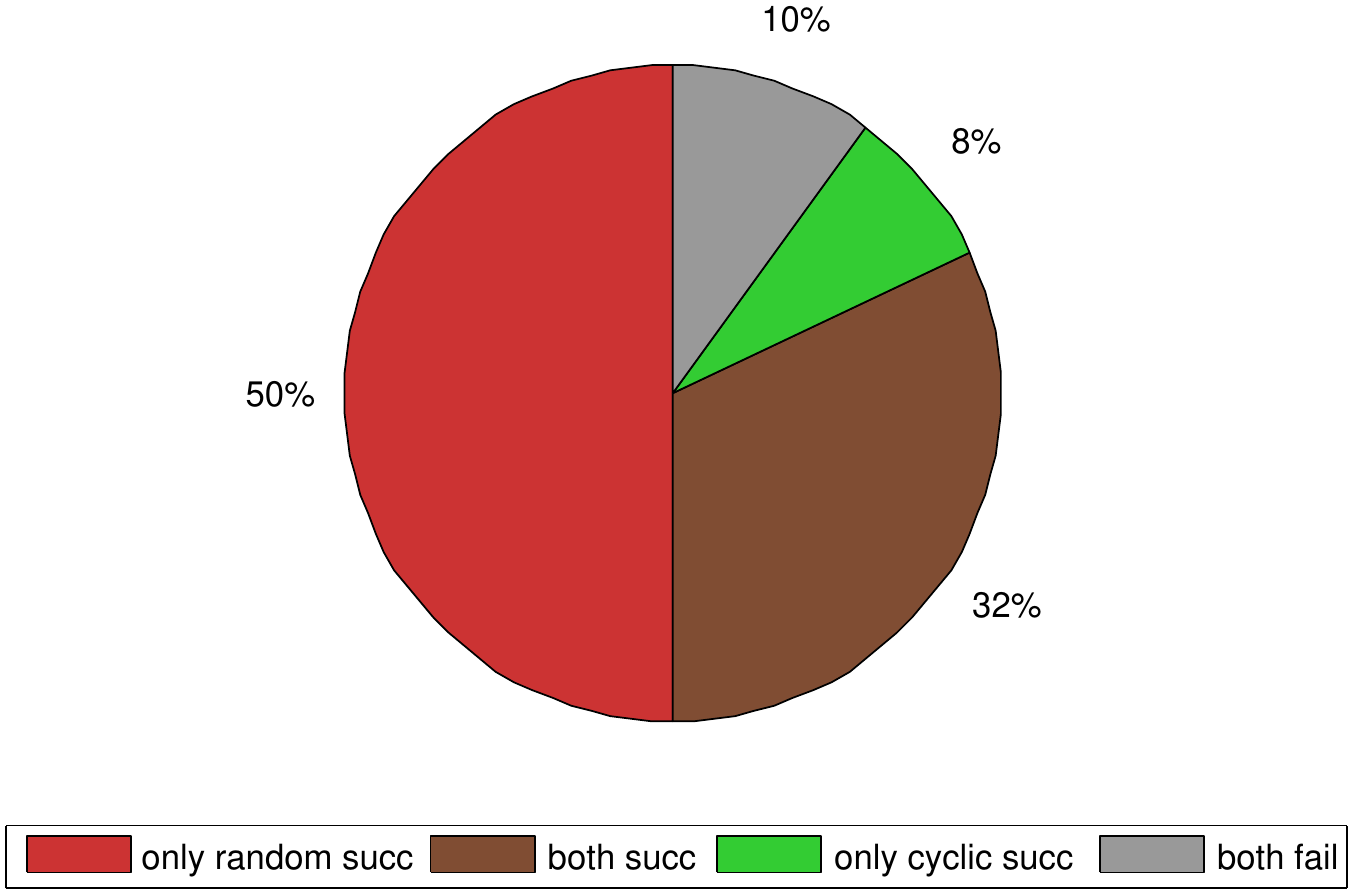}}
\end{tabular}}
\caption{All four cases of convergence behavior of the modified rank-one residue iteration \eqref{newupdate} with fixed cyclic order and with random shuffling. Both run to 100 cycles. The first plot implies both two versions fail and occurs 5 times among 50; the second plot implies both two versions succeed and occurs 16 times among 50; the third plot implies random version succeeds while the cyclic version fails and occurs 25 times among 50; the fourth plot implies cyclic version succeeds while the random version fails and occurs 4 times among 50.}\label{fig:result_swimmer}
\end{figure}

\subsection{Block prox-linear method for nonnegative Tucker decomposition} The nonnegative Tucker decomposition is to decompose a given nonnegative tensor (multi-dimensional array) into the product of a core nonnegative tensor and a few nonnegative factor matrices. It can be modeled as
\begin{equation}\label{ntd}
\min_{\bm{\cC}\ge0,\vA\ge0} \|\bm{\cC}\times_1\vA_1\ldots\times_N\vA_N-\bm{\cM}\|_F^2,
\end{equation}
where $\vA=(\vA_1,\ldots,\vA_N)$ and $\bm{\cX}\times_i\vY$ denotes tensor-matrix multiplication along the $i$-th mode (see \cite{kolda2009tensor} for example). The cyclic block proximal gradient method for solving \eqref{ntd} performs the following updates cyclically
\begin{subequations}\label{alg:bpg-ntd}
\begin{align}
\bm{\cC}^{k+1}=&\argmin_{\bm{\cC}\ge0}\langle\nabla_{\bm{\cC}}
f(\hat{\bm{\cC}}^k,\vA^k),\bm{\cC}-\hat{\bm{\cC}}^k\rangle+\frac{L_c^k}{2}\|\bm{\cC}-\hat{\bm{\cC}}^k\|_F^2,\\
\vA_i^{k+1}=&\argmin_{\vA_i\ge0}\langle\nabla_{\vA_i}f(\bm{\cC}^{k+1},\vA_{<i}^{k+1},\hat{\vA}_i^k,\vA_{>i}^k),\vA-\hat{\vA}^k\rangle+\frac{L_i^k}{2}\|\vA-\hat{\vA}^k\|_F^2,\,i=1,\ldots,N.
\end{align}
\end{subequations}
Here, $f(\bm{\cC},\vA)=\frac{1}{2}\|\bm{\cC}\times_1\vA_1\ldots\times_N\vA_N-\bm{\cM}\|_F^2$, $L_c^k$ and $L_i^k$ (chosen no less than a positive $L_{\min}$) are gradient Lipschitz constants with respect to $\bm{\cC}$ and $\vA_i$ respectively, and $\hat{\bm{\cC}}^k$ and $\hat{\vA}_i^k$ are extrapolated points:
\begin{equation}\label{ntd-ca}\hat{\bm{\cC}}^k=\bm{\cC}^k+\omega_c^k(\bm{\cC}^k-\bm{\cC}^{k-1}),\ \hat{\vA}_i^k=\vA_i^k+\omega_i^k(\vA_i^k-\vA_i^{k-1}),\, i = 1,\ldots N.
\end{equation}
with extrapolation weight set to
\begin{equation}\label{ntd-weight}\omega_c^k=\min\left(\omega_k, 0.9999\sqrt{\frac{L_c^{k-1}}{L_c^k}}\right),\ \omega_i^k=\min\left(\omega_k, 0.9999\sqrt{\frac{L_i^{k-1}}{L_i^k}}\right),\, i=1,\ldots N,
\end{equation}
where $\omega_k$ is the same as that in Algorithm \ref{alg:fista}. Our setting of extrapolated points exactly follows \cite{xu2015spntd}. Figure \ref{fig:ntd} shows that the extrapolation technique significantly accelerates the convergence speed of the method. Note that the block-prox method with no extrapolation reduces to the block coordinate gradient method in \cite{TsengYun2009}.
\begin{figure}
\centering
\includegraphics[width=0.4\textwidth]{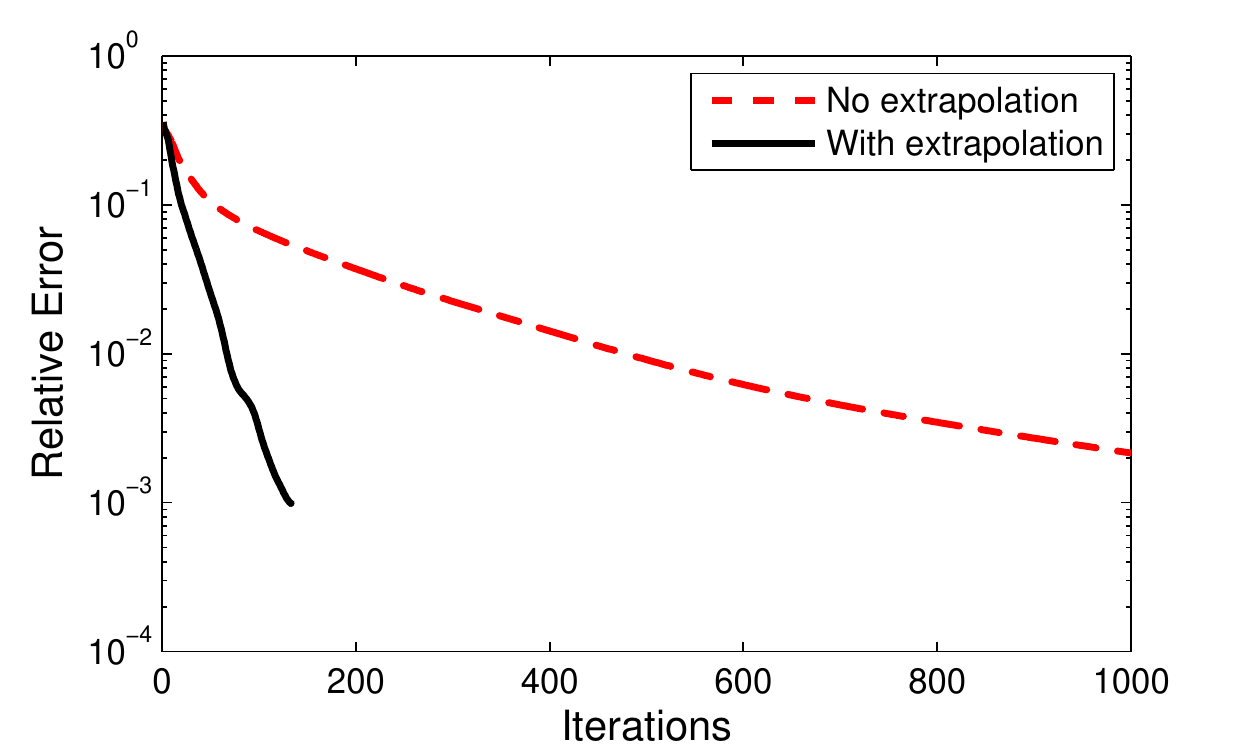}
\caption{Relative errors, defined as ${\|\bm{\cC}^k\times_1\vA_1^k\ldots\times_N\vA_N^k-\bm{\cM}\|_F}/{\|\bm{\cM}\|_F}$, given by \eqref{alg:bpg-ntd} on Gaussian randomly generated $80\times 80\times 80$ tensor with core size of $5\times 5\times 5$. No extrapolation: $\hat{\bm{\cC}}^k=\bm{\cC}^k, \hat{\vA}^k=\vA^k,\,\forall k$; With extrapolation: $\hat{\bm{\cC}}^k, \hat{\vA}^k$ set as in \eqref{ntd-ca} with extrapolation weights by \eqref{ntd-weight}.}\label{fig:ntd}
\end{figure}

Since the core tensor $\bm{\cC}$ interacts with all factor matrices, the work \cite{xu2015spntd} proposes to update $\bm{\cC}$ more frequently to improve the performance of the block proximal gradient method. Specifically, at each cycle, it performs the following updates sequentially from $i=1$ through $N$
\begin{subequations}\label{alg:ebpg-ntd}
\begin{align}
\bm{\cC}^{k+1,i}=&\argmin_{\bm{\cC}\ge0}\langle\nabla_{\bm{\cC}}
f(\hat{\bm{\cC}}^{k,i},\vA_{<i}^{k+1},\vA_{\ge i}^k),\bm{\cC}-\hat{\bm{\cC}}^{k,i}\rangle+\frac{L_c^{k,i}}{2}\|\bm{\cC}-\hat{\bm{\cC}}^{k,i}\|_F^2,\\
\vA_i^{k+1}=&\argmin_{\vA_i\ge0}\langle\nabla_{\vA_i}f(\bm{\cC}^{k+1,i},\vA_{<i}^{k+1},\hat{\vA}_i^k,\vA_{>i}^k),\vA-\hat{\vA}^k\rangle+\frac{L_i^k}{2}\|\vA-\hat{\vA}^k\|_F^2.
\end{align}
\end{subequations}
It was demonstrated that \eqref{alg:ebpg-ntd} numerically performs better than \eqref{alg:bpg-ntd}.
Numerically, we observed that the performance of \eqref{alg:ebpg-ntd} could be further improved if the blocks of variables were randomly shuffled as in \eqref{newupdate}, namely, we performed the updates sequentially from $i=1$ through $N$
\begin{subequations}\label{alg:rebpg-ntd}
\begin{align}
\bm{\cC}^{k+1,i}=&\argmin_{\bm{\cC}\ge0}\langle\nabla_{\bm{\cC}}
f(\hat{\bm{\cC}}^{k,i},\vA_{\pi^k_{<i}}^{k+1},\vA_{\pi^k_{\ge i}}^k),\bm{\cC}-\hat{\bm{\cC}}^{k,i}\rangle+\frac{L_c^{k,i}}{2}\|\bm{\cC}-\hat{\bm{\cC}}^{k,i}\|_F^2,\\
\vA_{\pi^k_i}^{k+1}=&\argmin_{\vA_{\pi^k_i}\ge0}\langle\nabla_{\vA_{\pi^k_i}}f(\bm{\cC}^{k+1,i},\vA_{\pi^k_{<i}}^{k+1},\hat{\vA}_{\pi^k_i}^k,\vA_{\pi^k_{>i}}^k),\vA-\hat{\vA}^k\rangle+\frac{L_i^k}{2}\|\vA-\hat{\vA}^k\|_F^2,
\end{align}
\end{subequations}
where $(\pi^k_1,\pi^k_2,\ldots,\pi^k_N)$ is a random permutation of  $(1,2,\ldots,N)$ at the $k$-th cycle. Note that both \eqref{alg:bpg-ntd} and \eqref{alg:rebpg-ntd} are special cases of Algorithm \ref{alg:ebpg} with $T=N+1$ and $T=2N+2$ respectively. If $\{(\bm{\cC}^k,\vA^k)\}$ is bounded, then so are $L_c^k, L_c^{k,i}$ and $L_i^k$'s. Hence, by Theorem \ref{thm:global-ebpg}, we have the convergence result as follows.

\begin{theorem}\label{thm:ntd}
The sequence $\{(\bm{\cC}^k,\vA^k)\}$ generated from \eqref{alg:bpg-ntd} or \eqref{alg:rebpg-ntd} is either unbounded or converges to a critical point of \eqref{ntd}.
\end{theorem}

We tested \eqref{alg:ebpg-ntd} and \eqref{alg:rebpg-ntd} on the $32\times32\times256$ Swimmer dataset used above and set the core size to $24\times 17\times 16$. We ran them to 500 cycles from the same random starting point. If the relative error $\|\bm{\cC}^{out}\times_1\vA_1^{out}\ldots\times_N\vA_N^{out}-\bm{\cM}\|_F/\|\bm{\cM}\|_F$ is below $10^{-3}$, we regard the decomposition to be successful, where $(\bm{\cC}^{out},\vA^{out})$ is the output. \emph{Among 50 independent runs, \eqref{alg:rebpg-ntd} with random shuffling succeeds 21 times while \eqref{alg:ebpg-ntd} succeeds only 11 times}. Figure \ref{fig:ntd_swimmer} plots all cases that occur. Similar to Figure \ref{fig:result_swimmer}, every plot is in terms of running time (sec), and during that time, both methods run to 500 iterations. From the figure, we see that \eqref{alg:rebpg-ntd} with fixed cyclic order and with random shuffling has similar computational complexity while the latter one can more frequently avoid bad local solutions.

\begin{figure}
\centering
{\footnotesize\begin{tabular}{cccc}
only random succeeds & both succeed & only cyclic succeeds & both fail\\
occurs 14/50 & occurs 7/50 & occurs 4/50 & occurs 25/50\\
\includegraphics[width=0.23\textwidth]{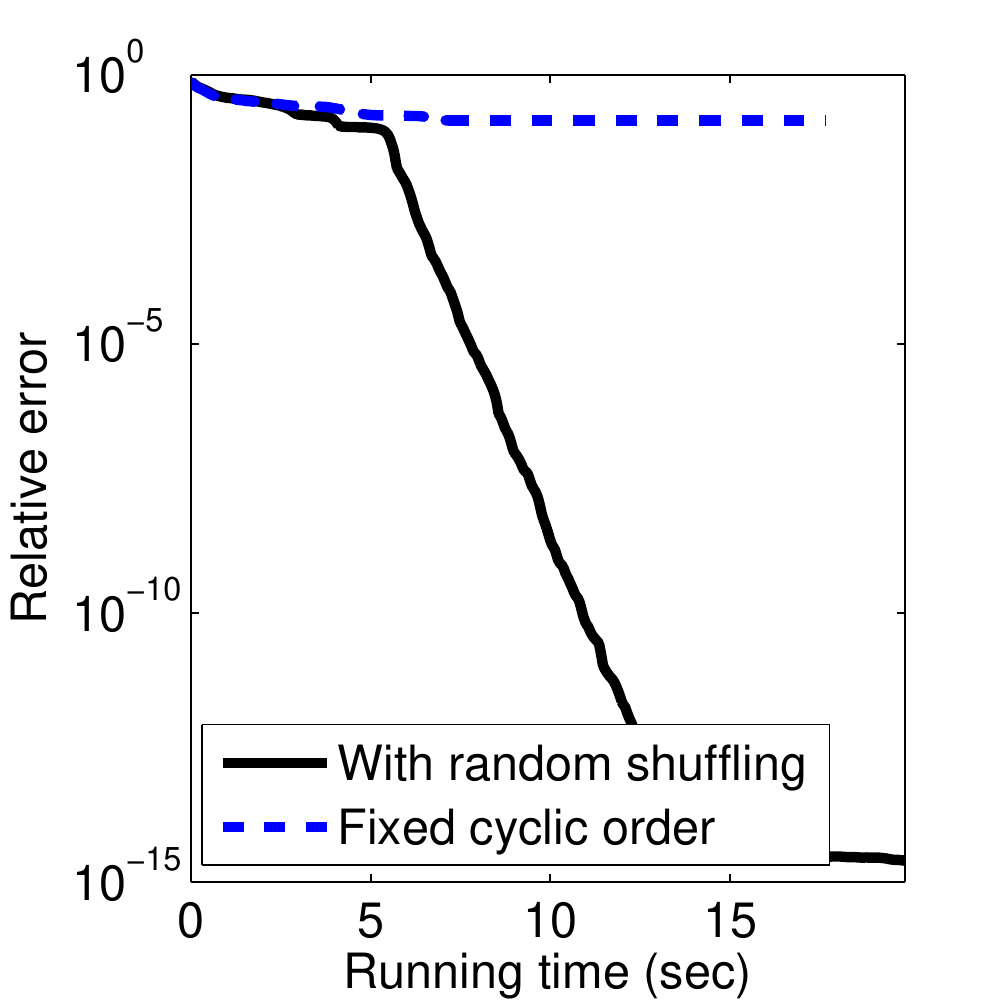}&
\includegraphics[width=0.23\textwidth]{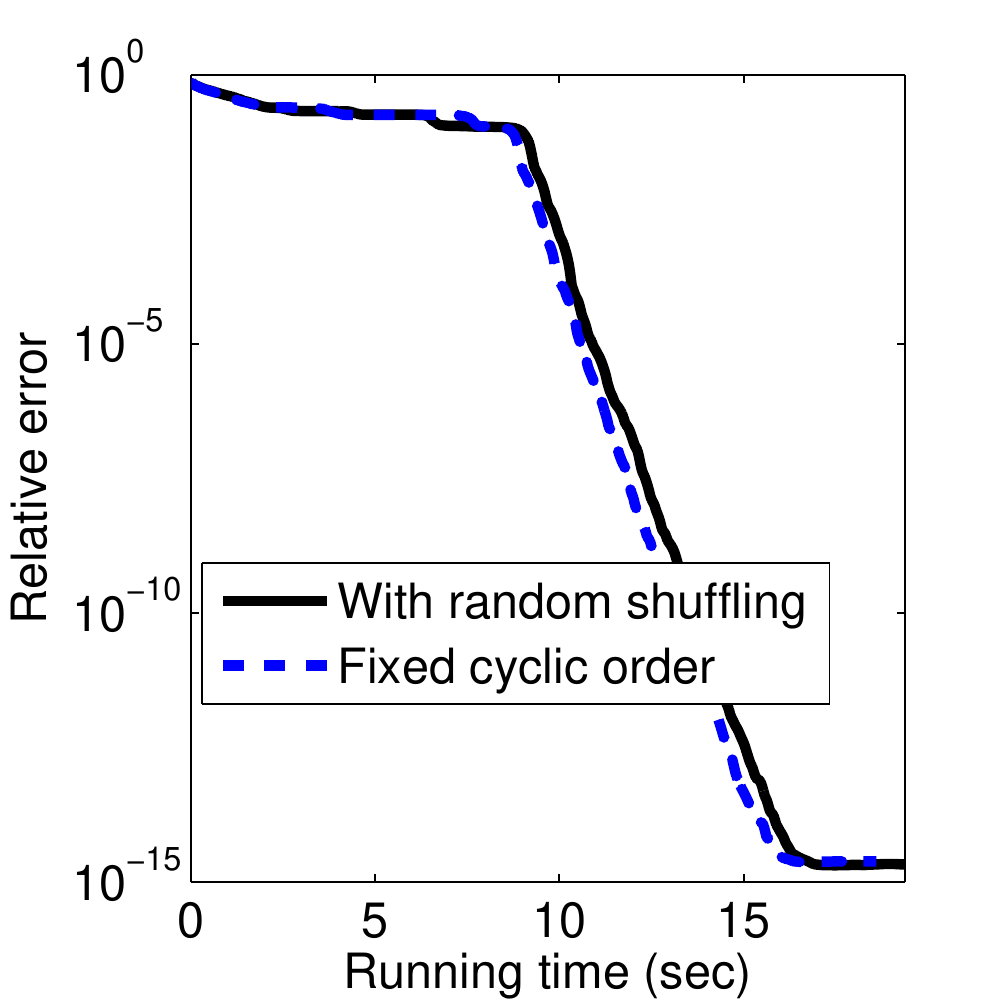}&
\includegraphics[width=0.23\textwidth]{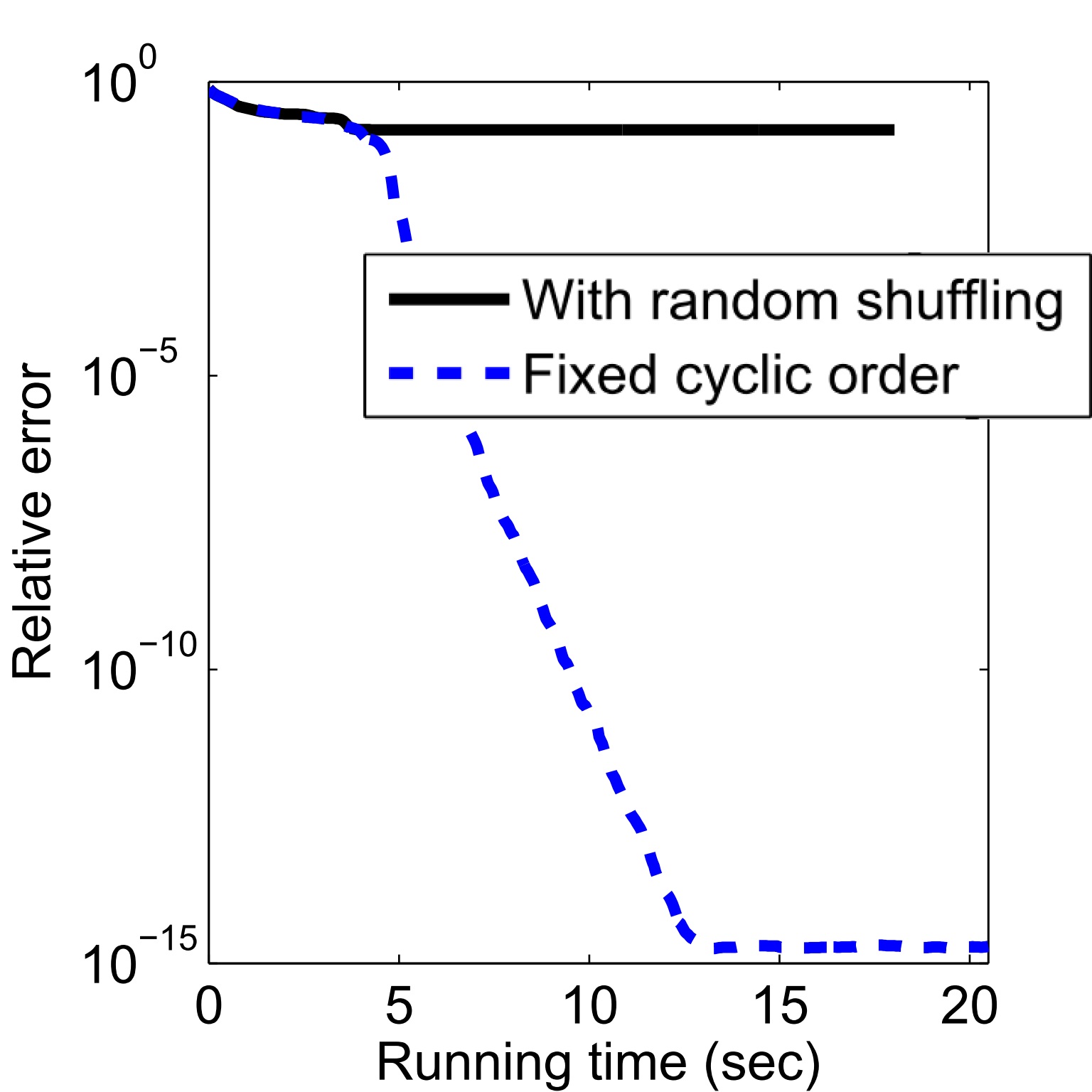}&
\includegraphics[width=0.23\textwidth]{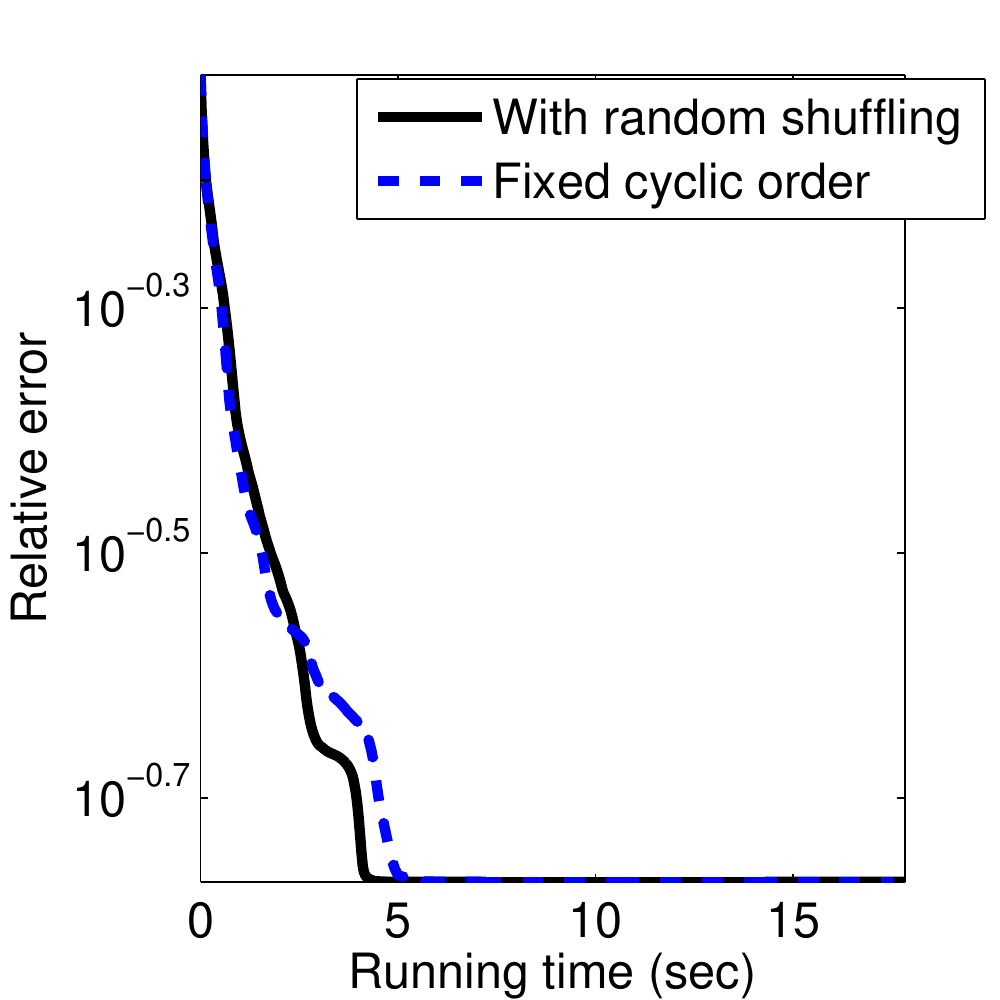}\\
\multicolumn{4}{c}{\includegraphics[width=0.45\textwidth]{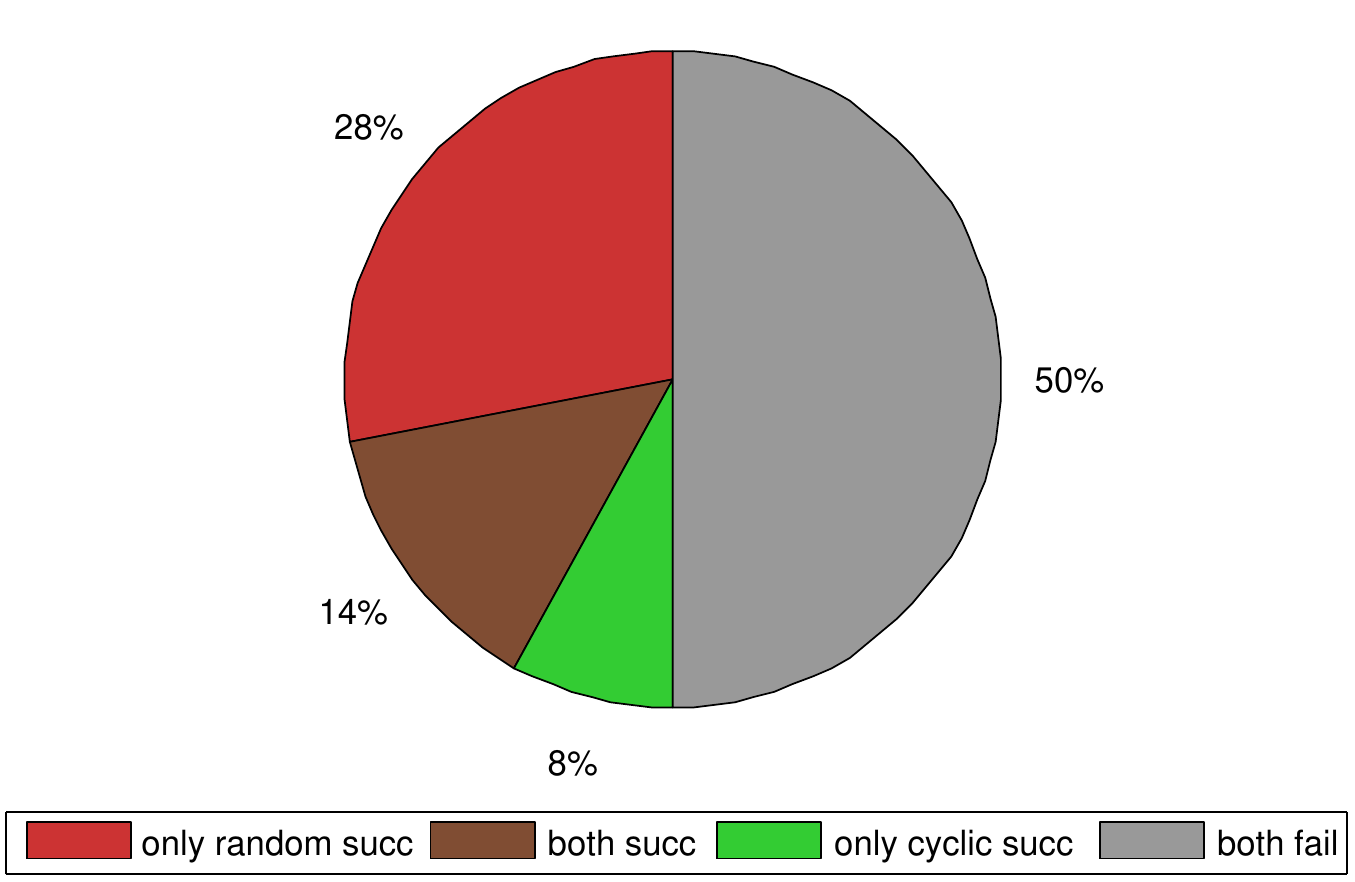}}
\end{tabular}}
\caption{All four cases of convergence behavior of the method \eqref{alg:rebpg-ntd} with fixed cyclic order and with random shuffling. Both run to 500 iterations. The first plot implies both two versions fail and occurs 25 times among 50; the second plot implies both two versions succeed and occurs 7 times among 50; the third plot implies random version succeeds while the cyclic version fails and occurs 14 times among 50; the fourth plot implies cyclic version succeeds while the random version fails and occurs 4 times among 50.}\label{fig:ntd_swimmer}
\end{figure}


\section{Conclusions}\label{sec:conclusion} We have presented a block prox-linear method, in both randomized and deterministic versions,  for solving nonconvex optimization problems. The method applies when the nonsmooth terms, if any, are  block separable. It is easy to implement and has a small memory footprint since only one block  is updated each time. 
Assuming that the differentiable parts have Lipschitz gradients, we showed that the method has a subsequence  of iterates that converges to a critical point. Further assuming the  Kurdyka-{\L}ojasiewicz property of the objective function, we showed that the entire sequence converges to a critical point and estimated its asymptotic convergence rate. Many applications have this property. In particular, we can apply our method and its convergence results to  $\ell_p$-(quasi)norm ($p\in[0,+\infty]$) regularized regression problems, matrix rank minimization, orthogonality constrained optimization, semidefinite programming, and so on. Very encouraging numerical results are presented.

\section*{Acknowledgements} The authors would like to thank three anonymous referees for their careful reviews and constructive comments.

\input{appendix.tex}
\bibliographystyle{siam}

\end{document}

%% file: macros.tex
\usepackage[normalem]{ulem} 

\newcommand{\bfx}{{\bf x}}

\newcommand{\bfy}{{\bf y}}

\newcommand{\bfu}{{\bf u}}
\newcommand{\bfv}{{\bf v}}

\newcommand{\va}{{\mathbf{a}}}
\newcommand{\vb}{{\mathbf{b}}}
\newcommand{\vc}{{\mathbf{c}}}

\newcommand{\ve}{{\mathbf{e}}}

\newcommand{\vg}{{\mathbf{g}}}
\newcommand{\vh}{{\mathbf{h}}}

\newcommand{\vm}{{\mathbf{m}}}

\newcommand{\vu}{{\mathbf{u}}}

\newcommand{\vx}{{\mathbf{x}}}
\newcommand{\vy}{{\mathbf{y}}}
\newcommand{\vz}{{\mathbf{z}}}

\newcommand{\vA}{{\mathbf{A}}}

\newcommand{\vE}{{\mathbf{E}}}

\newcommand{\vM}{{\mathbf{M}}}

\newcommand{\vX}{{\mathbf{X}}}
\newcommand{\vY}{{\mathbf{Y}}}

\newcommand{\cB}{{\mathcal{B}}}
\newcommand{\cC}{{\mathcal{C}}}

\newcommand{\cK}{{\mathcal{K}}}

\newcommand{\cM}{{\mathcal{M}}}
\newcommand{\cN}{{\mathcal{N}}}

\newcommand{\cX}{{\mathcal{X}}}
\newcommand{\cY}{{\mathcal{Y}}}

\newcommand{\RR}{\mathbb{R}} 
\newcommand{\vzero}{\mathbf{0}} 

\newcommand{\dist}{\mathrm{dist}}    
\newcommand{\dom}{{\mathrm{dom}}} 
\newcommand{\prox}{{\mathbf{prox}}} 
\DeclareMathOperator*{\argmin}{arg\,min} 
\DeclareMathOperator*{\argmax}{arg\,max} 

\newcommand{\prev}{{\mathrm{prev}}} 
\newcommand{\bm}[1]{\boldsymbol{#1}}

\newcommand{\st}{\mbox{ s.t. }}

\newcommand{\bc}{\begin{center}}
\newcommand{\ec}{\end{center}}

\newcommand{\bdm}{\begin{displaymath}}
\newcommand{\edm}{\end{displaymath}}

\newcommand{\beq}{\begin{equation}}
\newcommand{\eeq}{\end{equation}}

\newcommand{\bfl}{\begin{flushleft}}
\newcommand{\efl}{\end{flushleft}}

\newcommand{\bt}{\begin{tabbing}}
\newcommand{\et}{\end{tabbing}}

\newcommand{\beqn}{\begin{eqnarray}}
\newcommand{\eeqn}{\end{eqnarray}}

\newcommand{\beqs}{\begin{align*}} 
\newcommand{\eeqs}{\end{align*}}  



%% file: appendix.tex
\appendix
\section{Proofs of key lemmas}\label{app:proof-lem} In this section, we give proofs of the lemmas and also propositions we used.



\subsection{Proof of Lemma \ref{lem:dec0}} \label{app-1} We show the general case of $\alpha_k=\frac{1}{\gamma L_k},\forall k$ and $\tilde{\omega}_i^j\le\frac{\delta(\gamma-1)}{2(\gamma+1)}\sqrt{\tilde{L}_i^{j-1}/\tilde{L}_i^{j}},\,\forall i,j$.
Assume $b_k=i$. From the Lipschitz continuity of $\nabla_{\vx_i}f(\vx_{\neq i}^{k-1},\vx_i)$ about $\vx_i$, it holds that (e.g., see Lemma 2.1 in \cite{xu2013block})
\begin{equation}\label{ineq1-a}
f(\vx^{k})\le f(\vx^{k-1})+\langle\nabla_{\vx_i}f(\vx^{k-1}),\vx_i^{k}-\vx_i^{k-1}\rangle +\frac{L_k}{2}\|\vx_i^{k}-\vx_i^{k-1}\|^2.
\end{equation}
Since $\vx_i^{k}$ is the minimizer of \eqref{eq:ebpg}, then
\begin{equation}\label{ineq2-a}
\langle\nabla_{\vx_i}f(\vx_{\neq i}^{k-1},\hat{\vx}_i^k),\vx_i^{k}-\hat{\vx}_i^k\rangle +\frac{1}{2\alpha_k}\|\vx_i^{k}-\hat{\vx}_i^k\|^2+ r_i(\vx_i^{k})\le\langle\nabla_{\vx_i}f(\vx_{\neq i}^{k-1},\hat{\vx}_i^k),\vx_i^{k-1}-\hat{\vx}_i^k\rangle + \frac{1}{2\alpha_k}\|\vx_i^{k-1}-\hat{\vx}_i^k\|^2+r_i(\vx_i^{k-1}).
\end{equation} 
Summing \eqref{ineq1-a} and \eqref{ineq2-a} and noting that $\vx_j^{k+1}=\vx_j^k,\forall j\neq i$, we have
\begin{align}
&F(\vx^{k-1})-F(\vx^{k})\cr
=& f(\vx^{k-1})+r_i(\vx_i^{k-1})-f(\vx^k)-r_i(\vx_i^k)\cr
\ge & \langle\nabla_{\vx_i}f(\vx_{\neq i}^{k-1},\hat{\vx}_i^k)-\nabla_{\vx_i}f(\vx^{k-1}),\vx_i^{k}-{\vx}_i^{k-1}\rangle+\frac{1}{2\alpha_k}\|\vx_i^{k}-\hat{\vx}_i^k\|^2-\frac{1}{2\alpha_k}\|\vx_i^{k-1}-\hat{\vx}_i^k\|^2-\frac{L_k}{2}\|\vx_i^{k}-\vx_i^{k-1}\|^2\cr
= &\langle\nabla_{\vx_i}f(\vx_{\neq i}^{k-1},\hat{\vx}_i^k)-\nabla_{\vx_i}f(\vx^{k-1}),\vx_i^{k}-{\vx}_i^{k-1}\rangle+\frac{1}{\alpha_k}\langle\vx_i^{k}-\vx_i^{k-1},\vx_i^{k-1}-\hat{\vx}_i^k\rangle+(\frac{1}{2\alpha_k}-\frac{L_k}{2})\|\vx_i^{k}-\vx_i^{k-1}\|^2\cr
\ge &-\|\vx_i^{k}-\vx_i^{k-1}\|\big(\|\nabla_{\vx_i}f(\vx_{\neq i}^{k-1},\hat{\vx}_i^k)-\nabla_{\vx_i}f(\vx^{k-1})\|+\frac{1}{\alpha_k}\|\vx_i^{k-1}-\hat{\vx}_i^k\|\big)+(\frac{1}{2\alpha_k}-\frac{L_k}{2})\|\vx_i^{k}-\vx_i^{k-1}\|^2\cr
\ge &-\big(\frac{1}{\alpha_k}+L_k\big)\|\vx_i^{k}-\vx_i^{k-1}\|\cdot\|\vx_i^{k-1}-\hat{\vx}_i^k\|+(\frac{1}{2\alpha_k}-\frac{L_k}{2})\|\vx_i^{k}-\vx_i^{k-1}\|^2\cr
\overset{\eqref{eq:extrap}}= & -\big(\frac{1}{\alpha_k}+L_k\big)\omega_k\|\vx_i^{k}-\vx_i^{k-1}\|\cdot\|\vx_i^{k-1}-\tilde{\vx}_i^{d_i^{k-1}-1}\|+(\frac{1}{2\alpha_k}-\frac{L_k}{2})\|\vx_i^{k}-\vx_i^{k-1}\|^2\cr
\ge &\frac{1}{4}\big(\frac{1}{\alpha_k}-L_k\big)\|\vx_i^{k}-\vx_i^{k-1}\|^2-\frac{(1/\alpha_k+L_k)^2}{1/\alpha_k-L_k}\omega_k^2\|\vx_i^{k-1}-\tilde{\vx}_i^{d_i^{k-1}-1}\|^2\cr
=&\frac{(\gamma-1)L_k}{4}\|\vx_i^{k}-\vx_i^{k-1}\|^2-\frac{(\gamma+1)^2}{\gamma-1}L_k\omega_k^2\|\vx_i^{k-1}-\tilde{\vx}_i^{d_i^{k-1}-1}\|^2\nonumber.
\end{align}
Here, we have used Cauchy-Schwarz inequality in the second inequality, Lipschitz continuity of $\nabla_{\vx_i}f(\vx_{\neq i}^{k-1},\vx_i)$ in the third one, the Young's inequality in the fourth one, the fact $\vx_i^{k-1}=\tilde{\vx}_i^{d_i^k-1}$ to have the third equality, and $\alpha_k=\frac{1}{\gamma L_k}$ to get the last equality. Substituting $\tilde{\omega}_i^j\le\frac{\delta(\gamma-1)}{2(\gamma+1)}\sqrt{\tilde{L}_i^{j-1}/\tilde{L}_i^{j}}$ and recalling \eqref{seq} completes the proof. 

%

\subsection{Proof of the claim in Remark \ref{rm:multiconv}}\label{app-2} Assume $b_k=i$ and $\alpha_k=\frac{1}{L_k}$. When $f$ is block multi-convex and $r_i$ is convex, from Lemma 2.1 of \cite{xu2013block}, it follows that
\begin{align*}
& F(\vx^{k-1})-F(\vx^k)\\
\ge & \frac{L_k}{2}\|\vx_i^k-\hat{\vx}_i^k\|^2+L_k\langle\hat{\vx}_i^k-\vx_i^{k-1},\vx_i^k-\hat{\vx}_i^k\rangle\\
\overset{\eqref{eq:extrap}}= & \frac{L_k}{2}\|\vx_i^k-{\vx}_i^{k-1}-\omega_k({\vx}_i^{k-1}-{\vx}_i^{d_i^{k-1}-1})\|^2+L_k\omega_k\langle {\vx}_i^{k-1}-{\vx}_i^{d_i^{k-1}-1}, \vx_i^k-{\vx}_i^{k-1}-\omega_k({\vx}_i^{k-1}-{\vx}_i^{d_i^{k-1}-1})\rangle\\
=&  \frac{L_k}{2}\|\vx_i^k-{\vx}_i^{k-1}\|^2-\frac{L_k\omega_k^2}{2}\|\vx_i^{k-1}-{\vx}_i^{d_i^{k-1}-1}\|^2.
\end{align*}
Hence, if $\omega_k\le\delta\sqrt{\tilde{L}_i^{j-1}/\tilde{L}_i^j}$, we have the desired result.

\subsection{Proof of Proposition \ref{prop:sqsum}}\label{app:sqsum}
Summing \eqref{eq:diff} over $k$ from $1$ to $K$ gives
\begin{align*}
F(\bfx^0)-F(\bfx^K)\ge&\ \sum_{i=1}^s\sum_{k=1}^K\sum_{j=d_i^{k-1}+1}^{d_i^k}\left(\frac{\tilde{L}_i^j}{4}\|\tilde{\bfx}_i^{j-1}-\tilde{\bfx}_i^j\|^2-\frac{\tilde{L}_i^{j-1}\delta^2}{4}\|\tilde{\bfx}_i^{j-2}-\tilde{\bfx}_i^{j-1}\|^2\right)\\
=&\ \sum_{i=1}^s\sum_{j=1}^{d_i^{K}}\left(\frac{\tilde{L}_i^j}{4}\|\tilde{\bfx}_i^{j-1}-\tilde{\bfx}_i^j\|^2-\frac{\tilde{L}_i^{j-1}\delta^2}{4}\|\tilde{\bfx}_i^{j-2}-\tilde{\bfx}_i^{j-1}\|^2\right)\\
\ge &\ \sum_{i=1}^s\sum_{j=1}^{d_i^K}\frac{\tilde{L}_i^j(1-\delta^2)}{4}\|\tilde{\bfx}_i^{j-1}-\tilde{\bfx}_i^j\|^2\\
\ge &\ \sum_{i=1}^s\sum_{j=1}^{d_i^K}\frac{\ell(1-\delta^2)}{4}\|\tilde{\bfx}_i^{j-1}-\tilde{\bfx}_i^j\|^2,
\end{align*}
where we have used the fact $d_i^0=0,\forall i$ in the first equality, $\tilde{\bfx}_i^{-1}=\tilde{\bfx}_i^0,\forall i$ to have the second inequality, and $\tilde{L}_i^j\ge\ell, \forall i,j$ in the last inequality. Letting $K\to\infty$ and noting $d_i^K\to\infty$ for all $i$ by Assumption \ref{assump3}, we conclude from the above inequality and the lower boundedness of $F$ in Assumption \ref{assump1} that
$$\sum_{i=1}^s\sum_{j=1}^\infty\|\tilde{\vx}_i^{j-1}-\tilde{\vx}_i^j\|^2<\infty,$$
which implies \eqref{eq:sqsum}.

\subsection{Proof of Proposition \ref{prop-cvx}}\label{app:prop-cvx}
From Corollary 5.20 and Example 5.23 of \cite{rockafellar2009variational}, we have that if $\prox_{\alpha_kr_i}$ is single valued near $\vx_i^{k-1}-\alpha_k\nabla_{\vx_i}f(\vx^{k-1})$, then $\prox_{\alpha_kr_i}$ is continuous at $\vx_i^{k-1}-\alpha_k\nabla_{\vx_i}f(\vx^{k-1})$. Let $\hat{\vx}^k_i(\omega)$ explicitly denote the extrapolated point with weight $\omega$, namely, we take $\hat{\vx}^k_i(\omega_k)$ in \eqref{eq:extrap}. In addition, let $\vx^k_i(\omega)=\prox_{\alpha_kr_i}\big(\hat{\vx}_i^k(\omega)-\alpha_k\nabla_{\vx_i}f(\bfx_{\neq i}^{k-1},\hat{\bfx}_i^k(\omega))\big)$. Note that \eqref{eq:diff} implies 
\begin{equation}\label{pos-dec}F(\vx^{k-1})-F(\vx^k(0))\ge\|\vx^{k-1}-\vx^k(0)\|^2\overset{\eqref{not-opt}}>0.
\end{equation} 
From the optimality of $\vx^k_i(\omega)$, it holds that
\begin{align*}
&\langle\nabla_{\vx_i} f(\bfx_{\neq i}^{k-1},\hat{\bfx}_i^k(\omega)), \bfx_i^k(\omega)-\hat{\bfx}_i^k(\omega)\rangle+\frac{1}{2\alpha_k}\|\bfx_i^k(\omega)-\hat{\bfx}_i^k(\omega)\|^2+r_i(\bfx_i^k(\omega))\\
\le & \langle\nabla_{\vx_i} f(\bfx_{\neq i}^{k-1},\hat{\bfx}_i^k(\omega)), {\bfx}_i-\hat{\bfx}_i^k(\omega)\rangle+\frac{1}{2\alpha_k}\|{\bfx}_i-\hat{\bfx}_i^k(\omega)\|^2+r_i({\bfx}_i),\,\forall \vx_i.
\end{align*}
Taking limit superior on both sides of the above inequality, we have
\begin{align*}
&\langle\nabla_{\vx_i} f(\bfx^{k-1}), \bfx_i^k(0)-{\bfx}_i^{k-1}\rangle+\frac{1}{2\alpha_k}\|\bfx_i^k(0)-{\bfx}_i^{k-1}\|^2+\limsup_{\omega\to 0^+}r_i(\bfx_i^k(\omega))\\
\le & \langle\nabla_{\vx_i} f(\bfx^{k-1}), {\bfx}_i-{\bfx}_i^{k-1}\rangle+\frac{1}{2\alpha_k}\|{\bfx}_i-{\bfx}_i^{k-1}\|^2+r_i({\bfx}_i),\,\forall \vx_i,
\end{align*}
which implies $\underset{\omega\to 0^+}\limsup\, r_i(\bfx_i^k(\omega))\le r_i(\bfx_i^k(0))$. Since $r_i$ is lower semicontinuous, $\underset{\omega\to 0^+}\liminf\, r_i(\bfx_i^k(\omega))\ge r_i(\bfx_i^k(0))$. Hence, $\underset{\omega\to 0^+}\lim r_i(\bfx_i^k(\omega))= r_i(\bfx_i^k(0))$, and thus $\underset{\omega\to 0^+}\lim F(\bfx^k(\omega))= F(\bfx^k(0))$. Together with \eqref{pos-dec}, we conclude that there exists $\bar{\omega}_k>0$ such that $F(\vx^{k-1})-F(\bfx^k(\omega))\ge 0,\,\forall \omega\in[0,\bar{\omega}_k]$. This completes the proof.

\subsection{Proof of Lemma \ref{lem:seq}}\label{app:lem-seq}
Let $\va_m$ and $\vu_m$ be the vectors with their $i$-th entries
$$(\va_m)_i=\sqrt{\alpha_{i,n_{i,m}}},\quad (\vu_m)_i=A_{i,n_{i,m}}.$$
Then \eqref{cond-seq} can be written as
\begin{equation}\label{cond1}\|\va_{m+1}\odot\vu_{m+1}\|^2+(1-\beta^2)\sum_{i=1}^s\sum_{j=n_{i,m}+1}^{n_{i,m+1}-1}\alpha_{i,j}A_{i,j}^2
\le\beta^2\|\va_{m}\odot\vu_{m}\|^2+B_m\sum_{i=1}^s\sum_{j=n_{i,m-1}+1}^{n_{i,m}}A_{i,j}.
\end{equation}
 Recall
$$\underline{\alpha}=\inf_{i,j}\alpha_{i,j},\quad \overline{\alpha}=\sup_{i,j}\alpha_{i,j}.$$
Then it follows from \eqref{cond1} that
\begin{equation}\label{cond2}\|\va_{m+1}\odot\vu_{m+1}\|^2+\underline{\alpha}(1-\beta^2)\sum_{i=1}^s\sum_{j=n_{i,m}+1}^{n_{i,m+1}-1}A_{i,j}^2
\le\beta^2\|\va_{m}\odot\vu_{m}\|^2+B_m\sum_{i=1}^s\sum_{j=n_{i,m-1}+1}^{n_{i,m}}A_{i,j}.
\end{equation}
By the Cauchy-Schwarz inequality and noting $n_{i,m+1}-n_{i,m}\le N,\forall i,m$, we have
\begin{equation}\label{ineq-sq1}\left(\sum_{i=1}^s\sum_{j=n_{i,m}+1}^{n_{i,m+1}-1}A_{i,j}\right)^2\le sN\sum_{i=1}^s\sum_{j=n_{i,m}+1}^{n_{i,m+1}-1}A_{i,j}^2
\end{equation}
and for any positive $C_1$,
\begin{align}\label{ineq-sq2}
&(1+\beta)C_1\|\va_{m+1}\odot\vu_{m+1}\|\left(\sum_{i=1}^s\sum_{j=n_{i,m}+1}^{n_{i,m+1}-1}A_{i,j}\right)\cr
\le & \sum_{i=1}^s\sum_{j=n_{i,m}+1}^{n_{i,m+1}-1}\left(\frac{4-(1+\beta)^2}{4sN}\|\va_{m+1}\odot\vu_{m+1}\|^2+\frac{(1+\beta)^2C_1^2sN}{4-(1+\beta)^2}A_{i,j}^2\right)\cr
\le & \frac{4-(1+\beta)^2}{4}\|\va_{m+1}\odot\vu_{m+1}\|^2 + \frac{(1+\beta)^2C_1^2sN}{4-(1+\beta)^2}\sum_{i=1}^s\sum_{j=n_{i,m}+1}^{n_{i,m+1}-1}A_{i,j}^2.
\end{align}
Taking
\begin{equation}\label{eq-C1}C_1 \le\sqrt{\frac{\underline{\alpha}(1-\beta^2)(4-(1+\beta)^2)}{4sN}},
\end{equation} we have from \eqref{ineq-sq1} and \eqref{ineq-sq2} that
\begin{equation}\label{cond3}\frac{1+\beta}{2}\|\va_{m+1}\odot\vu_{m+1}\|+C_1\sum_{i=1}^s\sum_{j=n_{i,m}+1}^{n_{i,m+1}-1}A_{i,j}\le \sqrt{\|\va_{m+1}\odot\vu_{m+1}\|^2+\underline{\alpha}(1-\beta^2)\sum_{i=1}^s\sum_{j=n_{i,m}+1}^{n_{i,m+1}-1}A_{i,j}^2}.
\end{equation}
For any $C_2>0$, it holds
\begin{align}\label{cond4}
&\sqrt{\beta^2\|\va_{m}\odot\vu_{m}\|^2+B_m\sum_{i=1}^s\sum_{j=n_{i,m-1}+1}^{n_{i,m}}A_{i,j}}\cr
\le & \beta\|\va_{m}\odot\vu_{m}\| +\sqrt{B_m\sum_{i=1}^s\sum_{j=n_{i,m-1}+1}^{n_{i,m}}A_{i,j}}\cr
\le & \beta\|\va_{m}\odot\vu_{m}\|+C_2B_m+\frac{1}{4C_2}\sum_{i=1}^s\sum_{j=n_{i,m-1}+1}^{n_{i,m}}A_{i,j}\cr
\le & \beta\|\va_{m}\odot\vu_{m}\|+C_2B_m+\frac{1}{4C_2}\sum_{i=1}^s\sum_{j=n_{i,m-1}+1}^{n_{i,m}-1}A_{i,j}+\frac{\sqrt{s}}{4C_2}\|\vu_m\|.
\end{align}
Combining \eqref{cond2}, \eqref{cond3}, and \eqref{cond4}, we have
\begin{equation*}
\frac{1+\beta}{2}\|\va_{m+1}\odot\vu_{m+1}\|+C_1\sum_{i=1}^s\sum_{j=n_{i,m}+1}^{n_{i,m+1}-1}A_{i,j}\le \beta\|\va_{m}\odot\vu_{m}\|+C_2B_m+\frac{1}{4C_2}\sum_{i=1}^s\sum_{j=n_{i,m-1}+1}^{n_{i,m}-1}A_{i,j}+\frac{\sqrt{s}}{4C_2}\|\vu_m\|.
\end{equation*}
Summing the above inequality over $m$ from $M_1$ through $M_2\le M$ and arranging terms gives
\begin{align}\label{app-key1}
&\sum_{m=M_1}^{M_2}\left(\frac{1-\beta}{2}\|\va_{m+1}\odot\vu_{m+1}\|
-\frac{\sqrt{s}}{4C_2}\|\vu_{m+1}\|\right)+\big(C_1-\frac{1}{4C_2}\big)\sum_{m=M_1}^{M_2}\sum_{i=1}^s\sum_{j=n_{i,m}+1}^{n_{i,m+1}-1}A_{i,j}\cr
\le & \beta \|\va_{M_1}\odot\vu_{M_1}\|+C_2\sum_{m=M_1}^{M_2} B_m
+\frac{1}{4C_2}\sum_{i=1}^s\sum_{j=n_{i,M_1-1}+1}^{n_{i,M_1}-1}A_{i,j}+\frac{\sqrt{s}}{4C_2}\|\vu_{M_1}\|
\end{align}
Take 
\begin{equation}\label{eq-C2}
C_2=\max\left(\frac{1}{2C_1},\ \frac{\sqrt{s}}{\sqrt{\underline{\alpha}}(1-\beta)}\right).
\end{equation}
Then \eqref{app-key1} implies
\begin{align}\label{app-key2}
&\frac{\sqrt{\underline{\alpha}}(1-\beta)}{4}\sum_{m=M_1}^{M_2}\|\vu_{m+1}\|+\frac{C_1}{2}\sum_{m=M_1}^{M_2}\sum_{i=1}^s\sum_{j=n_{i,m}+1}^{n_{i,m+1}-1}A_{i,j}\cr
\le & \beta\sqrt{\overline{\alpha}}\|\vu_{M_1}\|+C_2\sum_{m=M_1}^{M_2} B_m+\frac{1}{4C_2}\sum_{i=1}^s\sum_{j=n_{i,M_1-1}+1}^{n_{i,M_1}-1}A_{i,j}+\frac{\sqrt{s}}{4C_2}\|\vu_{M_1}\|,
\end{align}
which together with 
$\sum_{i=1}^sA_{i,n_{i,m+1}}\le\sqrt{s}\|\vu_{m+1}\|$
gives
\begin{align}\label{cond5}
C_3\sum_{i=1}^s\sum_{j=n_{i,M_1}+1}^{n_{i,M_2+1}}A_{i,j}=&
C_3\sum_{m=M_1}^{M_2}\sum_{i=1}^s\sum_{j=n_{i,m}+1}^{n_{i,m+1}}A_{i,j}\cr
\le &\beta\sqrt{\overline{\alpha}}\|\vu_{M_1}\|+C_2\sum_{m=M_1}^{M_2} B_m+\frac{1}{4C_2}\sum_{i=1}^s\sum_{j=n_{i,M_1-1}+1}^{n_{i,M_1}-1}A_{i,j}+\frac{\sqrt{s}}{4C_2}\|\vu_{M_1}\|,\nonumber\\
\le & C_2\sum_{m=1}^{M_2} B_m+C_4\sum_{i=1}^s\sum_{j=n_{i,M_1-1}+1}^{n_{i,M_1}}A_{i,j},
\end{align}
where we have used $\|\vu_{M_1}\|\le \sum_{i=1}^sA_{i,n_{i,M_1}}$, and \begin{equation}\label{eq-C34}
C_3=\min\left(\frac{\sqrt{\underline{\alpha}}(1-\beta)}{4\sqrt{s}},\frac{C_1}{2}\right),\quad C_4=\beta\sqrt{\overline{\alpha}}+\frac{\sqrt{s}}{4C_2}.
\end{equation}
From \eqref{eq-C1}, \eqref{eq-C2}, and \eqref{eq-C34}, we can take
$$C_1=\frac{\sqrt{\underline{\alpha}}(1-\beta)}{2\sqrt{sN}}\le \min\left\{\sqrt{\frac{\underline{\alpha}(1-\beta^2)(4-(1+\beta)^2)}{4sN}},\ \frac{\sqrt{\underline{\alpha}}(1-\beta)}{2\sqrt{s}}\right\},
$$
where the inequality can be verified by noting $(1-\beta^2)(4-(1+\beta)^2)-(1-\beta)^2$ is decreasing with respect to $\beta$ in $[0,1]$. Thus from \eqref{eq-C2} and \eqref{eq-C34}, we have $C_2=\frac{1}{2C_1},\, C_3=\frac{C_1}{2},\, C_4=\beta\sqrt{\overline{\alpha}}+\frac{\sqrt{s}C_1}{2}$. 
Hence, from \eqref{cond5}, we complete the proof of \eqref{cond-seq0}.

If $\lim_{m\to\infty}n_{i,m}=\infty,\forall i$, $\sum_{m=1}^\infty B_m<\infty$, and \eqref{cond-seq} holds for all $m$, letting $M_1=1$ and $M_2\to\infty$, we have \eqref{cau-seq} from \eqref{cond5}.

\subsection{Proof of Proposition \ref{prop:asymconvg}}\label{app:asymconvg}
For any $i$, assume that while updating the $i$-th block to $\vx_i^k$, the value of the $j$-th block ($j\neq i$) is $\vy_j^{(i)}$, the extrapolated point of the $i$-th block is $\vz_i$, and the Lipschitz constant of $\nabla_{\vx_i}f(\vy_{\neq i}^{(i)},\vx_i)$ with respect to $\vx_i$ is $\tilde{L}_i$, namely,
$$\vx_i^k\in\argmin_{\vx_i}\langle \nabla_{\vx_i}f(\vy_{\neq i}^{(i)},\vz_i),\vx_i-\vz_i\rangle+\tilde{L}_i\|\vx_i-\vz_i\|^2+r_i(\vx_i).$$
Hence, $\vzero\in \nabla_{\vx_i}f(\vy_{\neq i}^{(i)},\vz_i)+2\tilde{L}_i(\vx_i^k-\vz_i)+\partial r_i(\vx_i^k),$
or equivalently,
\begin{equation}\label{firstorder}\nabla_{\vx_i}f(\vx^k)-\nabla_{\vx_i}f(\vy_{\neq i}^{(i)},\vz_i)-2\tilde{L}_i(\vx_i^k-\vz_i)\in\nabla_{\vx_i}f(\vx^k)+\partial r_i(\vx_i^k),\,\forall i.
\end{equation}

Note that $\vx_i$ may be updated to $\vx_i^k$ not at the $k$-th iteration but at some earlier one, which must be between $k-T$ and $k$ by Assumption \ref{assump3}. In addition, for each pair $(i,j)$, there must be some $\kappa_{i,j}$ between $k-2T$ and $k$ such that
\begin{equation}\label{yapp}
\vy_j^{(i)}=\vx_j^{\kappa_{i,j}},
\end{equation} 
and for each $i$, there are $k-3T\le\kappa_1^i<\kappa_2^i\le k$ and extrapolation weight $\tilde{\omega}_i\le 1$ such that
\begin{equation}\label{zapp}
\vz_i=\vx_i^{\kappa_2^i}+\tilde{\omega}_i(\vx_i^{\kappa_2^i}-
\vx_i^{\kappa_1^i}).
\end{equation}  
By triangle inequality, $(\vy_{\neq i}^{(i)},\vz_i)\in B_{4\rho}(\bar{\vx})$ for all $i$.
Therefore, it follows from \eqref{eq:cart-prod} and \eqref{firstorder} that
\begin{align}\label{cond6}\dist(\vzero,\partial F(\vx^k))\overset{\eqref{firstorder}}\le&\sqrt{\sum_{i=1}^s\|\nabla_{\vx_i}f(\vx^k)-\nabla_{\vx_i}f(\vy_{\neq i}^{(i)},\vz_i)-2\tilde{L}_i(\vx_i^k-\vz_i)\|^2}\cr
\le & \sum_{i=1}^s\|\nabla_{\vx_i}f(\vx^k)-\nabla_{\vx_i}f(\vy_{\neq i}^{(i)},\vz_i)-2\tilde{L}_i(\vx_i^k-\vz_i)\|\cr
\le &\sum_{i=1}^s\left(\|\nabla_{\vx_i}f(\vx^k)-\nabla_{\vx_i}f(\vy_{\neq i}^{(i)},\vz_i)\|+2\tilde{L}_i\|\vx_i^k-\vz_i\|\right)\cr
\le &\sum_{i=1}^s\left(L_G\|\vx^k-(\vy_{\neq i}^{(i)},\vz_i)\|+2\tilde{L}_i\|\vx_i^k-\vz_i\|\right)\cr
\le &\sum_{i=1}^s\left((L_G+2L)\|\vx_i^k-\vz_i\|+L_G\sum_{j\neq i}\|\vx_j^k-\vy_j^{(i)}\|\right),
\end{align}
where in the fourth inequality, we have used the Lipschitz continuity of $\nabla_{\vx_i}f(\vx)$ with respect to $\vx$, and the last inequality uses $\tilde{L}_i\le L$.
Now use \eqref{cond6}, \eqref{yapp}, \eqref{zapp} and also the triangle inequality to have the desired result.

\subsection{Proof of Lemma \ref{lem:rate-seq1}}\label{app:rate-seq1}
The proof follows that of Theorem 2 of \cite{attouch2009convergence}.
When $\gamma\ge 1$, since $0\le A_{k-1}-A_k\le 1,\forall k\ge K$, we have $(A_{k-1}-A_k)^\gamma\le A_{k-1}-A_k$, and thus \eqref{cond-seq1} implies that for all $k\ge K$, it holds that
$A_k\le (\alpha+\beta)(A_{k-1}-A_k)$, from which item 1 immediately follows.

When $\gamma<1$, we have $(A_{k-1}-A_k)^\gamma\ge A_{k-1}-A_k$, and thus \eqref{cond-seq1} implies that for all $k\ge K$, it holds that
$A_k\le (\alpha+\beta)(A_{k-1}-A_k)^\gamma$. Letting $h(x)=x^{-1/\gamma}$, we have for $k\ge K$,
\begin{align*}
1\le & (\alpha+\beta)^{1/\gamma}(A_{k-1}-A_k)A_k^{-1/\gamma}\\
= & (\alpha+\beta)^{1/\gamma}\left(\frac{A_{k-1}}{A_k}\right)^{1/\gamma}(A_{k-1}-A_k)A_{k-1}^{-1/\gamma}\\
\le &(\alpha+\beta)^{1/\gamma}\left(\frac{A_{k-1}}{A_k}\right)^{1/\gamma}\int_{A_k}^{A_{k-1}}h(x)dx\\
=&\frac{(\alpha+\beta)^{1/\gamma}}{1-1/\gamma}\left(\frac{A_{k-1}}{A_k}\right)^{1/\gamma}\left(A_{k-1}^{1-1/\gamma}-A_k^{1-1/\gamma}\right),
\end{align*}
where we have used nonincreasing monotonicity of $h$ in the second inequality. Hence,
\begin{equation}\label{cond-seq2}A_{k}^{1-1/\gamma}-A_{k-1}^{1-1/\gamma}\ge \frac{1/\gamma-1}{(\alpha+\beta)^{1/\gamma}}\left(\frac{A_{k}}{A_{k-1}}\right)^{1/\gamma}.
\end{equation}
Let $\mu$ be the positive constant such that
\begin{equation}\label{cond-seq3}\frac{1/\gamma-1}{(\alpha+\beta)^{1/\gamma}}\mu =\mu^{\gamma-1}-1.
\end{equation}
Note that the above equation has a unique solution $0<\mu<1$. We claim that \begin{equation}\label{cond-seq4}A_{k}^{1-1/\gamma}-A_{k-1}^{1-1/\gamma}\ge \mu^{\gamma-1}-1,\ \forall k\ge K.
\end{equation} It obviously holds from \eqref{cond-seq2} and \eqref{cond-seq3} if $\big(\frac{A_{k}}{A_{k-1}}\big)^{1/\gamma}\ge \mu$. It also holds if $\big(\frac{A_{k}}{A_{k-1}}\big)^{1/\gamma}\le \mu$ from the arguments
\begin{align*}
\left(\frac{A_{k}}{A_{k-1}}\right)^{1/\gamma}\le \mu \Rightarrow & A_k\le \mu^\gamma A_{k-1}\Rightarrow A_k^{1-1/\gamma}\ge\mu^{\gamma-1}A_{k-1}^{1-1/\gamma}\\
\Rightarrow & A_{k}^{1-1/\gamma}-A_{k-1}^{1-1/\gamma}\ge (\mu^{\gamma-1}-1)A_{k-1}^{1-1/\gamma} \ge \mu^{\gamma-1}-1,
\end{align*}
where the last inequality is from $A_{k-1}^{1-1/\gamma}\ge 1$. Hence, \eqref{cond-seq4} holds, and summing it over $k$ gives
$$A_k^{1-1/\gamma}\ge A_k^{1-1/\gamma}-A_K^{1-1/\gamma}\ge (\mu^{\gamma-1}-1)(k-K),$$
which immediately gives item 2 by letting $\nu=(\mu^{\gamma-1}-1)^{\frac{\gamma}{\gamma-1}}$.

\section{Solutions of \eqref{newupdate}}\label{app:newupdate} In this section, we give closed form solutions to both updates in \eqref{newupdate}. First, it is not difficult to have the solution of  \eqref{newupdate-y}: 
$$\vy_{\pi_i}^{k+1}=\max\left(0,\big(\vX_{\pi_{<i}}^{k+1}(\vY_{\pi_{<i}}^{k+1})^\top+\vX_{\pi_{>i}}^{k}(\vY_{\pi_{>i}}^{k})^\top-\vM\big)^\top\vx_{\pi_i}^{k+1}\right).$$ Secondly, since $L_{\pi_i}^k>0$, it is easy to write \eqref{newupdate-x} in the form of
$$\min_{\vx\ge0,\,\|\vx\|=1}\frac{1}{2}\|\vx-\va\|^2+\vb^\top\vx+C,$$
which is apparently equivalent to 
\begin{equation}\label{eqx}
\max_{\vx\ge0,\,\|\vx\|=1} \vc^\top\vx,
\end{equation}
which $\vc=\va-\vb$. Next we give solution to \eqref{eqx} in three different cases.

\textbf{Case 1: $\vc<0$.} Let $i_0=\argmax_i c_i$ and $c_{\max}=c_{i_0}<0$. If there are more than one components equal $c_{\max}$, one can choose an arbitrary one of them. Then the solution to \eqref{eqx} is given by $x_{i_0}=1$ and $x_i=0,\forall i\neq i_0$ because for any $\vx\ge0$ and $\|\vx\|=1$, it holds that
$$\vc^\top\vx\le c_{\max}\|\vx\|_1\le c_{\max}\|\vx\|=c_{\max}.$$

\textbf{Case 2: $\vc\le 0$ and $\vc\not<0$.} Let $\vc=(\vc_{I_0},\vc_{I_-})$ where $\vc_{I_0}=\vzero$ and $\vc_{I_-}<0$. Then the solution to \eqref{eqx} is given by $\vx_{I_-}=\vzero$ and $\vx_{I_0}$ being any vector that satisfies $\vx_{I_0}\ge0$ and $\|\vx_{I_0}\|=1$ because $\vc^\top\vx\le 0$ for any $\vx\ge 0$.

\textbf{Case 3: $\vc\not\le 0$.} Let $\vc=(\vc_{I_+},\vc_{I_+^c})$ where $\vc_{I_+}>0$ and $\vc_{I_+^c}\le 0$. Then \eqref{eqx} has a unique solution given by $\vx_{I_+}=\frac{\vc_{I_+}}{\|\vc_{I_+}\|}$ and $\vx_{I_+^c}=\vzero$ because for any $\vx\ge0$ and $\|\vx\|=1$, it holds that
$$\vc^\top\vx\le \vc_{I_+}^\top\vx_{I_+}\le \|\vc_{I_+}\|\cdot\|\vx_{I_+}\|\le \|\vc_{I_+}\|\cdot\|\vx\|=\|\vc_{I_+}\|,$$
where the second inequality holds with equality if and only if $\vx_{I_+}$ is collinear with $\vc_{I_+}$, and the third inequality holds with equality if and only if $\vx_{I_+^c}=\vzero$.

\section{Proofs of convergence of some examples}
In this section, we give the proofs of the theorems in section \ref{sec:example}.

%
%

\subsection{Proof of Theorem \ref{thm:rri}}\label{app:thm-rri}
Through checking the assumptions of Theorem \ref{thm:global-ebpg}, we only need to verify the boundedness of $\{\vY^k\}$ to show Theorem \ref{thm:rri}. Let $\vE^k=\vX^k(\vY^k)^\top-\vM$. Since every iteration decreases the objective, it is easy to see that $\{\vE^k\}$ is bounded. Hence, $\{\vE^k+\vM\}$ is bounded, and $$a=\sup_k\max_{i,j}(\vE^k+\vM)_{ij}<\infty.$$ Let $y_{ij}^k$ be the $(i,j)$-th entry of $\vY^k$. Thus the columns of $\vE^k+\vM$ satisfy
\begin{equation}\label{abound}
a\ge\ve_i^k+\vm_i=\sum_{j=1}^p y_{ij}^k \vx_j^k,\,\forall i,
\end{equation}
where $\vx_j^k$ is the $j$-th column of $\vX^k$. Since $\|\vx_j^k\|=1$, we have $\|\vx_j^k\|_\infty\ge 1/\sqrt{m},\,\forall j$. Note that \eqref{abound} implies each component of $\sum_{j=1}^p y_{ij}^k \vx_j^k$ is no greater than $a$. Hence from nonnegativity of $\vX^k$ and $\vY^k$ and noting that at least one entry of $\vx_j^k$ is no less than $1/\sqrt{m}$, we have $y_{ij}^k\le a\sqrt{m}$ for all $i,j$ and $k$. This completes the proof.